\documentclass[11pt]{amsart}
\usepackage{amsfonts}
\usepackage{amsmath}
\usepackage{amssymb}
\usepackage{multicol}
\usepackage{stmaryrd}
\usepackage{cite}
\usepackage{epsfig}
\usepackage{color}
\usepackage{graphics}
\usepackage{graphicx}
\usepackage{multicol,graphics}
\newcommand\bes{\begin{eqnarray}}
\newcommand\ees{\end{eqnarray}}
\newcommand\R{\mathbb R}
\newtheorem{theorem}{Theorem}[section]
\newtheorem{lemma}[theorem]{Lemma}

\newtheorem{proposition}[theorem]{Proposition}
\numberwithin{equation}{section}
\allowdisplaybreaks

\begin{document}

\title[Semi-wave and spreading speed ]{Semi-wave and spreading speed of the nonlocal Fisher-KPP equation with free boundaries}
\author[Y. Du, F. Li and M. Zhou]{Yihong Du$^\dag$, Fang Li$^\ddag$ and Maolin Zhou$^{\dag}$}
\thanks{\hspace{-.6cm}
$^\dag$ School of Science and Technology, University of New England, Armidale, NSW 2351, Australia.  \\
$^\ddag$  School of Mathematics, Sun Yat-sen University,   Guangzhou, Guangdong, 510275,  People's Republic of China.
}

\date{\today}

\maketitle

\begin{abstract}

In Cao, Du, Li and Li \cite{CDLL2019}, a nonlocal diffusion model with free boundaries extending the local diffusion model of Du and Lin \cite{DL2010} was introduced and studied.
For Fisher-KPP type nonlinearities,  its long-time dynamical behaviour is shown to follow  a spreading-vanishing dichotomy. However, when spreading happens, the question of spreading speed was left open in \cite{CDLL2019}.
In this paper we  obtain a rather complete answer to this question. We find a threshold condition on the kernel function such that  spreading grows linearly in time exactly when this condition holds, which is achieved by completely solving the associated semi-wave problem that determines this linear speed;  when the kernel function violates this condition, we show that accelerating spreading happens.
\bigskip


\hspace{1.3cm} Accelerating spreading.
\medskip


\end{abstract}

\section{Introduction}
\noindent

In this paper we continue the study of \cite{CDLL2019} on a ``nonlocal diffusion" version of the following
free boundary model with ``local diffusion'':
\begin{equation}
\left\{
\begin{aligned}
&u_t-du_{xx}=f(u),& &t>0,~g(t)<x<h(t),\\
&u(t,g(t))=u(t,h(t))=0,& &t>0,\\
&g'(t)=-\mu u_x(t,g(t)),\; h'(t)=-\mu u_x(t,h(t)),& &t>0,\\
&g(0)=g_0,\; h(0)=h_0,~u(0,x)=u_0(x),& &g_0\le x\le h_0,
\end{aligned}
\right.
\label{1}
\end{equation}
where $f$ is a $C^1$ function satisfying $f(0)=0$,
$\mu>0$ and $g_0<h_0$ are  constants, and $u_0$ is a $C^2$ function which is positive in $(g_0, h_0)$ and
vanishes at $x=g_0$ and $x=h_0$. For logistic type of $f(u)$, \eqref{1} was first studied in \cite{DL2010}, as a model for the spreading of a new or invasive species with population density $u(t,x)$, whose population range $(g(t), h(t))$ expands through its boundaries $x=g(t)$ and $x=h(t)$ according to the Stefan conditions
 $g'(t)=-\mu u_x(t, g(t)),\; h'(t)=-\mu u_x(t,h(t))$.
A deduction of these conditions based on some ecological assumptions can be found in \cite{BDK}.

It was shown in\cite{DL2010} that problem (\ref{1})
admits a unique solution $(u(t,x), g(t), h(t))$ defined for all $t>0$. Its long-time dynamical behaviour is characterised by a ``spreading-vanishing dichotomy'':
Either $(g(t), h(t))$ is contained in a bounded set of $\mathbb R$ for all $t>0$ and $u(t,x)\to 0$ uniformly as $t\to\infty$ (called the vanishing case),
or $(g(t), h(t))$ expands to $\mathbb R$ and $u(t,x)$ converges to the unique positive steady state of the ODE $v'=f(v)$ locally uniformly in $x\in\mathbb R$ as $t\to\infty$ (the spreading case). Moreover, when spreading occurs,
\[
\lim_{t\to\infty} \frac{-g(t)}{t}=\lim_{t\to\infty}\frac{h(t)}{t}=k_0>0,
\]
and $k_0$ is uniquely determined by a traveling wave equation associated to \eqref{1}.

These results have been extended to cases with more general $f(u)$ in \cite{DLou2015, KMY2019, KY2016}, and more accurate estimates for $g(t), h(t)$ and $u(t,x)$ for the spreading case have been obtained  in \cite{DMZ}.  Among the many further extensions, we only mention
the extension to various heterogeneous environments in \cite{DYH2013JFA, DyhLx2015, LLS16-1, LLS16-2, DDL-1, DDL-2, WangMX2016JFA}, to equations with an advection term \cite{GLZ, WZZ},  extensions to certain Lotka-Volterra two-species systems and epidemic models
 in \cite{DYH2014, DWZ, GJS2012, LZ2017, wangmx1}, and extensions to high space dimensions in \cite{DG2012, DMW2014}; see also the references therein.

Problem \eqref{1} is closely related to the following associated Cauchy problem
\begin{equation}
\left\{
\begin{aligned}
&U_t-dU_{xx}=f(U),& &t>0,~x\in\mathbb R,\\
&U(0,x)=U_0(x),& & x\in \mathbb R,
\end{aligned}
\right.
\label{Cauchy-local}
\end{equation}
Indeed, it follows from \cite{DG2012} that, if the initial functions are the same, i.e., $u_0=U_0$, then the unique solution $(u, g, h)$ of \eqref{1} and the unique solution $U$ of \eqref{Cauchy-local} are related in the following way: For any fixed $T>0$,
as $\mu\to\infty$, $(g(t), h(t))\to \mathbb R$ and $u(t,x)\to U(t,x)$ locally uniformly in $(t,x)\in(0, T]\times \mathbb R$. Thus \eqref{Cauchy-local} may be viewed as the limiting problem of \eqref{1} (as $\mu\to\infty$).

Problem \eqref{Cauchy-local} with $U_0$ a nonnegative function having nonempty compact support has long been used to describe the spreading of  a new or invasive species; see, for example, classical works \cite{Aronson1978, Fisher1937, KPP}. In both \eqref{1} and \eqref{Cauchy-local}, the dispersal of the species is described by the diffusion term
$d u_{xx}$, widely known as a ``local diffusion'' operator.

The nonlocal diffusion model with free boundaries considered in \cite{CDLL2019} has the following form:
\begin{equation}\label{main-fb}
\begin{cases}
\displaystyle u_t=d\int_{g(t)}^{h(t)}J(x-y)u(t,y)dy-du(t,x)+f(u),
& t>0,~x\in(g(t),h(t)),\\
\displaystyle u(t,g(t))=u(t,h(t))=0, &t>0,\\
\displaystyle h'(t)=\mu\int_{g(t)}^{h(t)}\int_{h(t)}^{+\infty}
J(x-y)u(t,x)dydx, &t>0,\\
\displaystyle g'(t)=-\mu\int_{g(t)}^{h(t)}\int_{-\infty}^{g(t)}
J(x-y)u(t,x)dydx, &t>0,\\
\displaystyle u(0,x)=u_0(x),~h(0)=-g(0)=h_0,  &  x\in[-h_0,h_0],
\end{cases}
\end{equation}
where $x=g(t)$ and $x=h(t)$ are the moving boundaries
to be determined together with $u(t,x)$, which is always assumed to be identically 0 for $x\in \R\setminus [g(t), h(t)]$; $d$ and $\mu$
 are positive constants. The initial
function $u_0(x)$ satisfies
\begin{equation}\label{initial-u0}
u_0\in C([-h_0,h_0]),~u_0(-h_0)=u_0(h_0)=0
~\text{ and }~u_0(x)>0~\text{ in }~(-h_0,h_0),
\end{equation}
with  $[-h_0,h_0]$ representing the initial population range of the species. The kernel function
$J: \mathbb{R}\rightarrow\mathbb{R}$  has the properties
\begin{description}
\item[(J)]$J\in C(\mathbb R)\cap L^\infty(\mathbb R),\; J\geq 0,\; J(0)>0,~\int_{\mathbb{R}}J(x)dx=1$, $J$  is even.
\end{description}
The growth term $f: \mathbb{R}^+\rightarrow\mathbb{R}$ is assumed to be continuous and satisfies
\begin{description}
\item[(f1)] $f(0) = 0$ and $f(u)$
is locally Lipschitz in $u\in\R^+$, i.e., for any $L>0$, there

 \hspace{0.1cm} exists a constant $K=K(L)>0$
such that
\[\hspace{0.8cm}
\mbox{$\left|f(u_1)-f(u_2)\right|\le K|u_1-u_2|$ for $u_1, u_2\in [0, L]$;}
\]
\end{description}
\begin{description}
\item[(f2)] There exists $K_0>0$ such that $f(u)<0$ for $u\ge K_0$.
\end{description}

The nonlocal free boundary problem \eqref{main-fb} may be viewed as describing the spreading of a new or invasive species with population density $u(t,x)$, whose population range $[g(t), h(t)]$ expands  according to the free boundary conditions
\[
\begin{cases}
\displaystyle h'(t)=\mu\int_{g(t)}^{h(t)}\int_{h(t)}^{+\infty}
J(x-y)u(t,x)dydx,\smallskip\\
\displaystyle g'(t)=-\mu\int_{g(t)}^{h(t)}\int_{-\infty}^{g(t)}
J(x-y)u(t,x)dydx,
\end{cases}
\]
that is,   the expanding rate of the range $[g(t), h(t)]$
is proportional to the outward flux of the population across the boundary of the range (see \cite{CDLL2019} for further explanations and justification).

One advantage of the nonlocal problem \eqref{main-fb} over the local problem \eqref{1} is that the nonlocal diffusion term 
\[
d\int_{g(t)}^{h(t)}J(x-y)u(t,y)dy-du(t,x)
\]
in \eqref{main-fb}
is capable to include spatial dispersal strategies of the species beyond random diffusion modelled by the term $d u_{xx}$ in \eqref{1}.

Under the assumptions {\bf (J),\ (f1)} and {\bf (f2)},
the well-posedness and global existence of (\ref{main-fb}) has been established in \cite{CDLL2019}.
If further, $f$ is a Fisher-KPP type function, namely it satisfies
\begin{description}
\item[(f3)] \ \ $f\in C^1$,\; $f>0=f(0) = f(1)$ in $(0,1)$, $f'(0)>0>f'(1)$, \mbox{ and }

 $ \hspace{0.4cm}  f(u)/u$ is nonincreasing \footnote{In {\bf (f3)} of \cite{CDLL2019}, it is assumed that $f(u)/u$ is strictly decreasing for $u>0$, but this request can be relaxed to $f(u)/u$ nonincreasing without affecting the conclusions there.} in $u>0$,
\end{description}
then the long-time dynamical behaviour of \eqref{main-fb} is determined by a ``spreading-vanishing dichotomy" (see Theorem 1.2 in \cite{CDLL2019}): As $t\to\infty$, either
\begin{itemize}
\item[(i)] \underline{\rm Spreading:} $\lim_{t\to+\infty} (g(t), h(t))=\mathbb R$ and $\lim_{t
\rightarrow+\infty}u(t,x)=1$ locally uniformly  in
$\mathbb{R}$, or
\item[(ii)] \underline{\rm Vanishing:} $\lim_{t\to+\infty} (g(t), h(t))=(g_\infty, h_\infty)$ is a finite interval and $\lim_{t
\rightarrow+\infty}u(t,x)=0$ uniformly for $x\in [g(t),h(t)]$.
\end{itemize}
Criteria for spreading and vanishing are also obtained in \cite{CDLL2019}; see Theorem 1.3 there. In particular, if the size of the initial population range $2h_0$  is large enough, then spreading always happens.

However, when spreading happens, the question of spreading speed was not considered
in \cite{CDLL2019}.
The main purpose of this paper is to determine the spreading speed left open there. 

In order to describe the main results of this paper, we introduce a key condition on the kernel function $J$, namely
\begin{description}
\item[(J1)] \ \  $\displaystyle \int_{-\infty}^0\int_0^{+\infty} J(x-y)dydx<+\infty$, \ i.e., $\displaystyle\int_{-\infty}^0\int_{-\infty} ^x J(y)dydx<+\infty$.
\end{description}

\begin{theorem} \label{thm1} Suppose that {\bf (J)} and {\bf (f3)} are satisfied, and spreading happens to the unique solution $(u, g, h)$ of \eqref{main-fb}. Then the following conclusions hold.
\begin{itemize}
\item[(i)] If {\bf (J1)} is satisfied, then there exists a unique $c_0>0$ such that
\[
\lim_{t\to\infty}\frac{h(t)}{t}=-\lim_{t\to\infty} \frac{g(t)}{t}=c_0.
\]
\item[(ii)] If {\bf (J1)} does not hold, then
\[
\lim_{t\to\infty}\frac{h(t)}{t}=-\lim_{t\to\infty} \frac{g(t)}{t}=+\infty.
\]
\end{itemize}
\end{theorem}

As usual, when {\bf (J1)} holds, we call $c_0$ the spreading speed of \eqref{main-fb}. The proof of Theorem \ref{thm1} relies on the  existence of semi-wave solutions to \eqref{main-fb}. These are pairs $(c,\phi)\in (0, +\infty)\times C^1((-\infty, 0])$ determined by the following two equations:
\begin{equation}\label{semi-wave}
\begin{cases}
\displaystyle d \int_{-\infty}^0 J(x-y) \phi(y) dy - d \phi(x)+ c\phi'(x) + f(\phi(x)) =0, &  -\infty < x< 0,\\
\displaystyle \phi(-\infty) = 1,\ \ \phi(0) =0,
\end{cases}
\end{equation}
and
\begin{equation}\label{fby-0}
c=\mu\int_{-\infty}^{0}\int_{0}^{+\infty}J(x-y)\phi(x)dydx.
\end{equation}

If $(c,\phi)$ solves \eqref{semi-wave}, then we call $\phi$ a semi-wave with speed $c$, since  the function $v(t,x): = \phi (x - ct)$ satisfies
\begin{equation*}
\begin{cases}
\displaystyle v_t = d \int_{-\infty}^{ct} J(x-y) v(t,y) dy - d v(t,x ) + f(v(t,x)), &  t>0, \  x<ct,\\
\displaystyle v(t, -\infty) =1, \ \  v(t, ct) =0, & t>0.\end{cases}
\end{equation*}
However, only the semi-wave satisfying \eqref{fby-0}  meets the free boundary condition along the moving front $x=ct$, and hence useful for determining the long-time dynamical behaviour of \eqref{main-fb}.

\begin{theorem}\label{thm2}
Suppose that {\bf (J)} and {\bf (f3)} are satisfied. Then \eqref{semi-wave}-\eqref{fby-0} has a solution pair $(c, \phi)=(c_0,\phi^{c_0})\in (0, +\infty)\times C^1((-\infty, 0])$ with $\phi^{c_0}(x)$ nonincreasing in $x$ if and only if {\bf (J1)} holds. Moreover, when {\bf (J1)} holds, there exists a unique such solution pair, and 
$\phi^{c_0}(x)$ is strictly decreasing in $x$.
\end{theorem}

The uniquely determined $c_0>0$ in Theorem \ref{thm2} is the spreading speed for \eqref{main-fb} given in part (i) of Theorem \ref{thm1}.

\bigskip

To put these results into perspective, we now recall some related results for the corresponding nonlocal diffusion problem of \eqref{Cauchy-local}, namely
\begin{equation}\label{Cauchy}
\begin{cases}
\displaystyle u_t=d\int_{\mathbb R}J(x-y)u(t,y)dy-du(t,x)+f(u),
& t>0,~x\in\mathbb R,\\
\displaystyle u(0,x)=u_0(x), &  x\in\mathbb R.
\end{cases}
\end{equation}
Problem \eqref{Cauchy} and its many variations have been extensively studied in the literature; see, for example, \cite{AC, Bates1997, BatesJMAA2007, BCV2016, BGHP, ChenXF1997, CDM2013, CD, FT, Garnier, Hutson2003JMB, KLS, RSZ, SLW, Yagisita2009} and the references therein. In particular, if {\bf (J)} and {\bf (f3)} are satisfied, and if the nonnegative initial function $u_0$ has non-empty compact support, then the basic long-time dynamical behaviour of \eqref{Cauchy} is given by
\[
\lim_{t\to\infty} u(t,x)=1 \ \ \mbox{ locally uniformly for $x\in\mathbb R$}.
\]
To understand the fine spreading behaviour of \eqref{Cauchy}, one examines the level set
\[
E_\lambda(t):=\{x\in\mathbb R: u(t,x)=\lambda\} \mbox{ with fixed } \lambda\in (0,1),
\]
by considering the large time behaviour of
\[
x_\lambda^+(t):=\sup E_\lambda(t) \ \ \mbox{ and } \;\; x_\lambda^-(t)=\inf E_\lambda(t).
\]

For this purpose, the following additional condition on the kernel function, apart from {\bf (J)},  is important:
\begin{description}
\item[(J2)] \ \ There exists $\lambda>0$ such that
\[
\int_{-\infty}^{+\infty} J(x)e^{\lambda x}dx<\infty.
\]
\end{description}

Yagisita \cite{Yagisita2009} has proved the following result on traveling wave solutions to \eqref{Cauchy}:
\begin{proposition}\label{tw}
Suppose that $f$ satisfies {\bf (f3)} and $J$ satisfies {\bf (J)}. If additionally $J$ satisfies {\bf (J2)},  then there is a constant $c_*>0$ such that \eqref{Cauchy}
has a traveling wave solution with speed $c$ if and only if $c\geq c_*$.  To be more precise,  the problem
\begin{equation}\label{tw-cauchy}
\begin{cases}
\displaystyle d \int_{\mathbb R} J(x-y) \phi(y) dy- d\phi(x) + c \phi'(x) + f(\phi(x)) =0,  &  x\in\mathbb R,\\
\phi(-\infty) =1,\ \phi(+\infty) =0
\end{cases}
\end{equation}
has a solution $\phi\in L^\infty(\mathbb R)$ which is nonincreasing if and only if $c\geq c_*$. Moreover, for each $c\geq c_*$, the solution has the following property \footnote{This follows easily from the proof in \cite{Yagisita2009}.}: $\phi\in C^1(\mathbb R)$. On the other hand, if  $J$ does not satisfy {\bf (J2)},
then \eqref{Cauchy} does not have a traveling wave solution, that is, for any constant $c$, \eqref{tw-cauchy}
has no solution $\phi\in L^\infty(\mathbb R)$ which is nonincreasing \footnote{Theorem 2 in \cite{Yagisita2009} actually provides a stronger nonexistence result.}.
\end{proposition}

Condition {\bf (J2)} is often called a ``thin tail" condition for $J$, and if it is not satisfied, then $J$ is said to have a ``fat tail". When $f$ satisfies {\bf (f3)}, and $J$ satisfies {\bf (J)} and {\bf (J2)}, it is well known (see, for example, \cite{W82}) that
\bes\label{c*}
\lim_{t\to\infty}\frac{|x_\lambda^{\pm}(t)|}{t}=c_*,
\ees
with $c_*$ given by Proposition \ref{tw}. On the other hand, if {\bf (f3)} and {\bf (J)} hold but {\bf (J2)} is not satisfied, then it follows from Theorem 6.4 of \cite{W82}  that $|x_\lambda^{\pm}(t)|$ grows faster than any linear function of $t$ as $t\to\infty$, namely,
\[
\lim_{t\to\infty}\frac{|x_\lambda^{\pm}(t)|}{t}=\infty.
\]
 Such a behaviour is usually called ``accelerating spreading". See also \cite{AC, BGHP, CR2013, FF, FT, Garnier, HR2010} and references therein for further progress on this and related questions.

We can easily show  that {\bf (J2)} implies {\bf (J1)}. Indeed, from
\[
a(x):=\int_0^\infty J(x-y)dy=\int_{-\infty}^xJ(z)dz=\int_{-x}^\infty J(z)dz \mbox{ for } x\leq 0,
\]
we obtain
\[
a(x)\leq e^{-\lambda |x|} \int_{|x|}^\infty J(z)e^{\lambda z}dz\leq e^{-\lambda|x|}\int_{-\infty}^\infty J(z)e^{\lambda z}dz,
\]
which clearly implies \[
\int_{-\infty}^0a(x)dx<\infty,
\]
i.e., {\bf (J1)} holds. On the other hand, it is easily checked that $J(x)=(1+x^2)^{-\sigma}$ with $\sigma>1$ satisfies {\bf (J1)} but not {\bf (J2)}.

Therefore, for $f$ satisfying {\bf (f3)}, there exist kernel functions $J$ satisfying {\bf (J)} and {\bf (J1)} but not {\bf (J2)} such that the free boundary problem \eqref{main-fb} spreads linearly with speed $c_0$, but the corresponding problem \eqref{Cauchy} has accelerating spreading.

Theorem \ref{thm1} indicates that for \eqref{main-fb}, under conditions {\bf (f3)} and {\bf (J)}, accelerating spreading happens exactly when {\bf (J1)} is not satisfied.
In sharp contrast, let us recall that, for \eqref{1}, which is the corresponding local diffusion problem of \eqref{main-fb}, when spreading happens, the spreading speed is always finite; see \cite{BDK, DL2010, DLou2015, DMZ}.

More can be said about the relationship between \eqref{main-fb} and \eqref{Cauchy}. For the local diffusion versions of \eqref{main-fb} and \eqref{Cauchy}, namely \eqref{1} and \eqref{Cauchy-local}, it is known that as $\mu\to+\infty$, the spreading speed of \eqref{1} converges to the spreading speed of \eqref{Cauchy-local} (see \cite{DLou2015}), and moreover, it follows from \cite{DG2012} that the Cauchy problem \eqref{Cauchy-local} can be viewed as the limiting problem of the free boundary problem \eqref{1} as $\mu\to+\infty$. Here we show that similar results hold for the nonlocal diffusion problems \eqref{main-fb} and \eqref{Cauchy}; see Theorems \ref{thm5.1}, \ref{thm5.2} and \ref{thm5.3} for details.

The rest of this paper is organized as follows. In Section 2, we prove Theorem \ref{thm2} on the semi-wave solution, which paves the ground for this research. In Section 3 we prove the first part of Theorem \ref{thm1}, by making use of the semi-wave  solution established in Section 2. Section 4 is devoted to the proof of the second part of Theorem \ref{thm1}, on accelerating spreading, by making use of the first part of the theorem proved in Section 3 and an approximation argument.
In Section 5, we consider the limiting profile of \eqref{main-fb} and its semi-wave solution as $\mu\to+\infty$.

\section{A unique semi-wave with the desired speed}

The purpose of this section is to prove Theorem \ref{thm2}. We first prove the existence  of a family of semi-waves to \eqref{main-fb}, namely for any speed $c$ in a certain range, there exists a unique  positive solution  $\phi=\phi^c\in C^1((-\infty, 0])$ to the problem \eqref{semi-wave}.
We will further show that for any given $\mu>0$, there exists a unique $c=c_0$ satisfying
\begin{equation}\label{fby}
c=\mu\int_{-\infty}^{0}\int_{0}^{+\infty}J(x-y)\phi^{c}(x)dydx.
\end{equation}

As we will see later, the unique solution $\phi^{c_0}(x)$ is strictly decreasing in $x$ for $x\in (-\infty, 0]$, and hence from the condition {\bf (J1)}, it is easily seen that
\begin{equation}\label{hat-c}
c_0<\mu c(J) \ \ \ \ \mbox{ with }\ \  c(J):=\int_{-\infty}^0\int_0^{+\infty} J(x-y)dydx.
\end{equation}

In the following, we will first prove the existence and uniqueness of $(c_0, \phi^{c_0})$ for the case that $J$ satisfies {\bf (J2)}, and then use an approximation argument to show that the conclusion also holds when $J$ satisfies {\bf (J1)}. It is
 easy to show that {\bf (J1)} is a necessary condition for the existence of such a pair $(c_0, \phi^{c_0})$.

\subsection{A perturbed problem} Suppose that {\bf (J2)} holds. Fix $\sigma\in (0, 1)$, $c\in (0, c_*)$, and consider the auxiliary problem
\begin{equation}\label{semi-wave-sigma}
\begin{cases}
\displaystyle d \int_{-\infty}^{+\infty} J(x-y) \phi(y) dy - d \phi(x)+ c\phi'(x) + f(\phi(x)) =0, &  -\infty < x< 0,\\
\displaystyle \phi(-\infty) = 1,\ \ \phi(x) =\sigma, & 0\leq x<+\infty.
\end{cases}
\end{equation}
We will show that \eqref{semi-wave-sigma} has a solution $\phi_\sigma$, which converges to the unique solution of \eqref{semi-wave} as $\sigma\to 0$.

If $\phi$ solves \eqref{semi-wave-sigma}, then clearly, for $x<0$,
\[
-c\phi'(x)= d \int_{-\infty}^0 J(x-y) \phi(y) dy +d \int_0^{+\infty} J(x-y) \sigma dy - d \phi(x) + f(\phi(x)).
\]
Choose $M>0$ large so that
\[
\mbox{ $u\mapsto \tilde f(u):=(cM-d)u+f(u)$ is increasing for $u\in [0,1]$,}
\]
and denote
\[
a(x)=\int_0^{+\infty}J(x-y)dy=\int_{-\infty}^x J(y)dy.
\]
 Then for $x<0$,
\[
-c(e^{-Mx}\phi)'=e^{-Mx}\left[d\!\! \int_{-\infty}^0 J(x-y) \phi(y) dy +d \sigma a(x)+  \tilde f(\phi(x))\right],
\]
and hence
\[
\phi(x)=e^{Mx}\sigma+\frac{e^{Mx}}{c}\int_x^0\!\!e^{-M\xi}\left[d\!\! \int_{-\infty}^0\!\! J(\xi-y) \phi(y) dy +d \sigma a(\xi)+ \tilde f(\phi(\xi))\right]d\xi.
\]

We now define an operator $A$ over
\[
\Omega:=\big\{\phi\in C(\mathbb R): 0\leq \phi\leq 1\big\}
\]
by
\[
A[\phi](x)=\left\{\begin{array}{ll}\displaystyle e^{Mx}\sigma+\frac{e^{Mx}}{c}\!\int_x^0\!\!e^{-M\xi}\left[d\!\! \int_{-\infty}^0\!\! J(\xi-y) \phi(y) dy +d \sigma a(\xi)
+ \tilde f(\phi(\xi))\right]d\xi, & x<0,\\
\sigma,& x\geq 0.
\end{array}\right.
\]
Then $\phi\in\Omega$ solves \eqref{semi-wave-sigma} if and only if $\phi$ is a fixed point of $A$ in $\Omega$.

Let $\phi_*$ denote a traveling wave solution with minimal speed $c_*$ given by Proposition \ref{tw}. Clearly $\phi_*\in\Omega$, and by a suitable translation we may assume that
\[
\phi_*(0)=\sigma.
\]
If its dependence on $\sigma$ need to be stressed, we will write $\phi_*(x)=\phi_{*\sigma}(x)$.
Then define
\[
\phi_*^{\sigma}(x)=\max\{\phi_*(x),\sigma\}=\max\{\phi_{*\sigma}(x),\sigma\}=\left\{\begin{array}{ll} \phi_{*\sigma}(x), & x<0,\\
\sigma,& x\geq 0.
\end{array}\right.
\]
We show next that
\begin{equation}\label{ls}
A[\phi_*^{\sigma}](x)\geq \phi_*^{\sigma}(x),\ \ \  A[1](x)<1\; \mbox{ for \ } x\in\mathbb R.
\end{equation}
Evidently $A[\phi_*^{\sigma}](x)=\phi_*^{\sigma}(x)=\sigma$ for $x\geq 0$. For $x<0$, we have
\begin{equation*}
\begin{array}{ll}
\displaystyle A[\phi_*^{\sigma}](x) &\displaystyle= e^{Mx}\sigma+\frac{e^{Mx}}{c}\!\int_x^0\!\!e^{-M\xi}\left[d\!\! \int_{-\infty}^0\!\! J(\xi-y) \phi_*(y) dy +d \sigma a(\xi)+ \tilde f(\phi_*(\xi))\right]d\xi\smallskip\\
&\geq \displaystyle e^{Mx}\sigma+\frac{e^{Mx}}{c}\!\int_x^0\!\!e^{-M\xi}\left[d\!\! \int_{-\infty}^{+\infty}\!\! J(\xi-y) \phi_*(y) dy + \tilde f(\phi_*(\xi))\right]d\xi \smallskip\\
&=\displaystyle e^{Mx}\sigma+\frac{e^{Mx}}{c}\!\int_x^0\!\!e^{-M\xi}\big[cM\phi_*(\xi)-c_*\phi_*'(\xi)\big]d\xi \smallskip\\
&>\displaystyle e^{Mx}\sigma+\frac{e^{Mx}}{c}\!\int_x^0\!\!e^{-M\xi} \big[cM\phi_*(\xi)-c\phi_*'(\xi)\big]d\xi \smallskip\\
&= \displaystyle e^{Mx}\sigma-e^{Mx}\int_x^0[e^{-M\xi} \phi_*(\xi)]'d\xi= \phi_*(x)=\phi_*^{\sigma}(x).
\end{array}
\end{equation*}
Therefore the first inequality in \eqref{ls} holds. To prove the second inequality, again we only need to check it for $x<0$, where we have
\begin{equation*}
\begin{array}{ll}
\displaystyle A[1](x) &\displaystyle= e^{Mx}\sigma+\frac{e^{Mx}}{c}\!\int_x^0\!\!e^{-M\xi}\left[d\!\! \int_{-\infty}^0\!\! J(\xi-y) dy +d \sigma a(\xi) + \tilde f(1)\right]d\xi\smallskip\\
&< \displaystyle e^{Mx}\sigma+\frac{e^{Mx}}{c}\!\int_x^0\!\!e^{-M\xi}\left[d\!\! \int_{-\infty}^{+\infty}\!\! J(\xi-y)  dy  + \tilde f(1))\right]d\xi \smallskip\\
&=\displaystyle e^{Mx}\sigma+\frac{e^{Mx}}{c}\!\int_x^0\!\!e^{-M\xi}cMd\xi \smallskip\\
&= \displaystyle 1+(\sigma-1)e^{Mx}<1.
\end{array}
\end{equation*}
This proves \eqref{ls}.

We now define inductively
\[
\phi_0(x)=\phi_*^{\sigma}(x),\; \phi_{n+1}(x)=A[\phi_n](x)=A^n[\phi_*^{\sigma}](x),\; n=0,1,2,..., x\in\mathbb R.
\]
 The monotonicity of $\tilde f$ implies that the operator $A$ is monotone increasing, namely
\[
\phi, \tilde\phi\in\Omega \mbox{ and } \phi\leq \tilde\phi \mbox{ imply } A[\phi](x)\leq A[\tilde\phi](x).
\]
Using this property of $A$ and  \eqref{ls}
we obtain
\[
\phi_0(x)\leq  \phi_{n}(x)\leq \phi_{n+1}(x)<1 \mbox{ for }  n=1,2,..., x\in\mathbb R.
\]
We now define
\[
\phi_\sigma(x):=\lim_{n\to\infty} \phi_n(x).
\]
Clearly $\phi_\sigma(x)=\sigma$ for $x\geq 0$, and for $x<0$, by the Lebesque dominated convergence theorem, we deduce from $\phi_{n+1}(x)=A[\phi_n](x)$ that
\[
\phi_\sigma(x)=A[\phi_\sigma](x).
\]
Since $\phi_*^{\sigma}(x)=\phi_0(x)\leq \phi_\sigma(x)\leq 1$ and $\phi_0(-\infty)=1$, we necessarily have $\phi_\sigma(-\infty)=1$. From the expression of $A[\phi_\sigma](x)$ and $\phi_\sigma(x)=A[\phi_\sigma](x)$, we see that $\phi_\sigma'(x)$ exists and is continuous for $x<0$, and hence $\phi=\phi_\sigma$
satisfies \eqref{semi-wave-sigma}.

We have thus proved the following conclusion.

\begin{lemma} Suppose that {\bf (J2)} holds. Then
for any $\sigma\in (0,1)$ and $c\in (0, c_*)$,
 \eqref{semi-wave-sigma} has a solution $\phi_\sigma$, which can be obtained by an iteration process.
\end{lemma}

We show next that the $\phi_\sigma(x)$ obtained in this way is nonincreasing  in $x$ and nondecreasing in $\sigma$, at least for all small $\sigma>0$.

\begin{lemma}\label{lem:mono} There exists $\delta_0\in (0,1]$ such that
\[
\begin{cases}
\phi_\sigma(x)\geq \phi_\sigma(y) & \mbox{ if } x\leq y\leq 0,\; \sigma\in (0, \delta_0);\\
 \phi_{\sigma_1}(x)\leq \phi_{\sigma_2}(x) &\mbox{ if } 0<\sigma_1\leq \sigma_2<1,\; x\leq 0.
 \end{cases}
\]
\end{lemma}
\begin{proof} 
Clearly $\phi_*^\sigma$ is monotone increasing in $\sigma$. It follows that
\[
A^n[\phi_*^{\sigma_1}](x)\leq A^n[\phi_*^{\sigma_2}](x) \mbox{ for } x\leq 0,\; 0<\sigma_1\leq \sigma_2<1,\; n=1,2,... 
\]
Letting $n\to\infty$ we obtain $ \phi_{\sigma_1}(x)\leq \phi_{\sigma_2}(x) \mbox{ if } 0<\sigma_1\leq \sigma_2<1,\; x\leq 0$. 

To show the monotonicity in $x$, it suffices to show the monotonicity in $x$ for every $\phi_n(x):=A^n[\phi_{n-1}](x)$, with $\phi_0(x)=\phi_*^{\sigma}(x)$.
We observe that from the definition of $A$ we know that each $\phi_n(x)$ is differentiable in $x$ for $x<0$. Moreover, we already proved that $\phi_n$ is increasing in $n$.
We next show by an induction argument that $\phi_n'(x)< 0$ for all $n\geq 0$, $x<0$ and $\sigma\in (0,\delta_0)$, with
\[
\delta_0:=\min\{\frac{\sigma_0}{Mc_*}, 1\}.
\]
Clearly this holds for $n=0$. Suppose that $\phi_k'(x)<0$ for $x<0$ and some nonnegative integer $k$. We show that $\phi_{k+1}'(x)<0$ for $x<0$ and $\sigma\in (0, \delta_0)$.

By definition, we have
\[
\phi_{k+1}(x)= \sigma e^{Mx}+\frac {e^{Mx}}{c}\int_x^0e^{-M\xi}g_k(\xi)d\xi,
\]
with 
\begin{align*}
g_k(\xi):=& d\int_{-\infty}^0 J(\xi-y)\phi_k(y)dy+d\sigma a(\xi)+\tilde f(\phi_k(\xi))\\
=& d\int_{-\infty}^{-\xi}J(z)\phi_k(z+\xi)dz+ d\sigma a(\xi)+\tilde f(\phi_k(\xi)).
\end{align*}
Clearly
\[
g_k(0)\geq d \int_{-\infty}^0 J(z)\phi_k(z)dz\geq \sigma_0:=d\int_{-\infty}^0J(z)\phi_*(z)dz.
\]
Moreover,
\begin{align*}
g_k'(\xi)=&-dJ(-\xi)\phi_k(0)+d\int_{-\infty}^{-\xi}J(z)\phi_k'(z+\xi)dz+d\sigma a'(\xi)+\tilde f'(\phi_k(\xi))\phi_k'(\xi)\\
\leq & -dJ(-\xi)\sigma+d\sigma J(\xi)=0.
\end{align*}
We thus obtain
\begin{align*}
\phi_{k+1}'(x)=&\sigma Me^{Mx}+M\frac{e^{Mx}}{c}\int_x^0 e^{-M\xi}g_k(\xi)d\xi-\frac 1c g_k(x)\\
=& \sigma Me^{Mx}+M\frac{e^{Mx}}{c}\left[ -\frac{e^{-M\xi}}{M}g_k(\xi)\Big|_x^0+\int_x^0 \frac{e^{-M\xi}}{M}g_k'(\xi)d\xi\right]-\frac 1c g_k(x)\\
\leq & \sigma Me^{Mx}+M\frac{e^{Mx}}{c}\left[-\frac{g_k(0)}{M}+\frac{e^{-Mx}}{M}g_k(x) \right]-\frac 1c g_k(x)\\
=& e^{Mx}[M\sigma-\frac 1 c g_k(0)]< e^{Mx}[M\sigma-\frac{\sigma_0}{c_*}]<0
\end{align*}
since $0<\sigma<\delta_0:=\min\{\frac{\sigma_0}{Mc_*}, 1\}$.
\end{proof}
\begin{lemma}\label{mono-c}
If $0<c_1<c_2<c_*$ and $\phi_\sigma^i$ denotes the solution obtained from the iteration process with $c=c_i$, $i=1,2$, then
$\phi_\sigma^1\geq \phi_\sigma^2$ for $\sigma\in (0, \delta_0)$.
\end{lemma}
\begin{proof} To stress the dependence of the operator $A$ on $c$, we
denote it by $A_c$.
For $x<0$, from $(\phi_\sigma^1)'(x)\leq 0$ proved in the previous lemma, and
\[
-c_1(\phi_\sigma^1)'(x)=d\int_{-\infty}^0J(x-y)\phi_\sigma^1(y)dy+d\sigma a(x)+f(\phi_\sigma^1(x)),
\]
we obtain
\[
-c_2(\phi_\sigma^1)'(x)\geq d\int_{-\infty}^0J(x-y)\phi_\sigma^1(y)dy+d\sigma a(x)+f(\phi_\sigma^1(x)).
\]
It follows that
$\phi_\sigma^1(x)\geq A_{c_2}[\phi_\sigma^1](x)$. Using this and $\phi_\sigma^1\geq \tilde \phi_*$ we deduce
\[
\phi_\sigma^1(x)\geq A_{c_2}^n[\phi_\sigma^1](x)\geq A_{c_2}^n[\tilde\phi_*](x),\;\; n=1,2,....
\]
Letting $n\to\infty$ we obtain $\phi_\sigma^1\geq \phi_\sigma^2$, as desired.
\end{proof}

\subsection{Existence of semi-waves and spatial monotonicity}

In this subsection, we make use of $\phi_\sigma$ obtained above to prove  the following result.
\begin{theorem}\label{thm-semiwave-existence} Suppose that {\bf (J2)} holds. Then
for every $c\in (0, c_*)$, there exists a  semi-wave $\phi^c$ to the problem \eqref{semi-wave}  with speed $c$, and $\phi^c(x)$ is nonincreasing for $x\in(-\infty, 0]$.
\end{theorem}
\begin{proof}

Let $\sigma_n$ be a decreasing sequence in $(0,\delta_0)\cap (0, 1/2)$ satisfying $\sigma_n\to 0$ as $n\to\infty$, where $\delta_0$ is given in Lemma 2.2. Then due to the monotonicity of $\phi_{\sigma_n}(x)$ in $x$, and the fact that
$\phi_{\sigma_n}(-\infty)=1,\; \phi_{\sigma_n}(0)=\sigma_n\in (0, 1/2)$, there exists a unique $\tilde x_n<0$ such that
\[
\phi_{\sigma_n}(\tilde x_n)=1/2,\; \phi_{\sigma_n}(x)<1/2 \mbox{ for } x>\tilde x_n.
\]
By Lemma \ref{lem:mono} we easily deduce $\tilde x_m \leq \tilde x_n$ for $m>n$.  Set
$$
\tilde \phi_n (x)  = \phi_{\sigma_n} (x + \tilde x_n)\ \ \textrm{for} \  x< - \tilde x_n.
$$
Then $\tilde \phi_n $ satisfies,  for $x< - \tilde x_n$
\begin{equation}\label{tilde-phi-n}
d \int_{-\infty}^{- \tilde x_n}  J(x-y)  \tilde \phi_n  (y)dy  + d \int_{- \tilde x_n}^{+\infty}   J(x-y)  \sigma_n  dy  - d \tilde \phi_n + c \tilde \phi_{n}' + f(\tilde \phi_n) =0.
\end{equation}
In view of $\tilde x_m \leq \tilde x_n<0$ for  $m> n$, there are two possible cases:
\begin{itemize}
\item{\it Case 1.} $- \tilde x_n \rightarrow + \infty$ as $n\rightarrow  +\infty$.
\item{\it Case 2.} $- \tilde x_n \rightarrow  x_0$ as $n\rightarrow  +\infty$ for some $x_0\in(0,+\infty)$.
\end{itemize}

Since  $\tilde \phi_n$ and by the equation subsequently $\tilde\phi_n'$ are uniformly bounded,  by the Arzela-Ascoli Theorem and a standard argument involving a diagonal process of choosing subsequences, there exist $\tilde \phi_{\infty}\in C(\mathbb R)$ and a subsequence of $\{ \tilde \phi_n \}_{n\geq 1}$, still denoted by $\{ \tilde \phi_n \}_{n\geq 1}$, such that  $\tilde \phi_n $ converges to $\tilde \phi_{\infty}$ locally uniformly in $\mathbb R$. (Here we extend $\tilde\phi_n(x)$ by $\sigma_n$ for $x\geq -\tilde x_n$.) Moreover, $\tilde\phi_\infty(x)$ is nonincreasing in $x$, and
$
 \tilde \phi_{\infty} (0) = {1\over 2}.
$

Moreover, if {\it Case 1}   happens,  we can verify that  $\tilde \phi_{\infty}$ satisfies
\begin{equation}\label{tilde-phi-infty}
d \int_{-\infty}^{\infty }  J(x-y)  \tilde \phi_{\infty}  (y)dy     - d \tilde \phi_{\infty} + c \tilde \phi_{\infty}' + f(\tilde \phi_{\infty}) =0.
\end{equation}
Indeed, by  (\ref{tilde-phi-n}),  we have
\begin{eqnarray}\label{tilde-phi-int}
c \left( \tilde \phi_{n}(x) - {1\over 2} \right) &=&  - d \int_0^x \int_{-\infty}^{- \tilde x_n}  J(z-y)  \tilde \phi_n  (y)dy dz -d \int_0^x \int_{- \tilde x_n}^{\infty}   J(z-y)  \sigma_n  dy dz\cr
&&+ d \int_0^x \tilde \phi_n (z) dz  -\int_0^x  f(\tilde \phi_n(z)) dz.
\end{eqnarray}
Fix $x\in\mathbb R$;  by the dominated convergence theorem, one easily sees that by letting $n\rightarrow \infty$, the above equation yields
\begin{eqnarray*}
c  \left( \tilde \phi_{\infty}(x) - {1\over 2} \right) &=&  - d \int_0^x \int_{-\infty}^{ \infty}  J(z-y)  \tilde \phi_{\infty}  (y)dy dz   \\
&&+ d \int_0^x \tilde \phi_{\infty }(z) dz  -\int_0^x  f(\tilde \phi_{\infty}(z)) dz,
\end{eqnarray*}
and thus (\ref{tilde-phi-infty}) follows by differetiating this equation. However, (\ref{tilde-phi-infty}) contradicts to the fact that $c_*$ is the minimal speed. Hence {\it Case 1} cannot happen.

Therefore {\it Case 2} must happen. Similarly  fix $x\in\mathbb R$ and let $n\rightarrow \infty$ in (\ref{tilde-phi-int}); we obtain
\begin{eqnarray*}
c  \left( \tilde \phi_{\infty}(x) - {1\over 2} \right) &=&  - d \int_0^x \int_{-\infty}^{x_0}  J(z-y)  \tilde \phi_{\infty}  (y)dy dz   \\
&&+ d \int_0^x \tilde \phi_{\infty }(z) dz  -\int_0^x  f(\tilde \phi_{\infty}(z)) dz,
\end{eqnarray*}
which yields
$$
\begin{cases}
\displaystyle  d \int_{-\infty}^{x_0 }  J(x-y)  \tilde \phi_{\infty}  (y)dy     - d \tilde \phi_{\infty} + c \tilde \phi_{\infty}'+ f(\tilde \phi_{\infty}) =0,  &   {x<  x_0}, \\
\displaystyle  \ \ \ \ \ \ \  \ \tilde \phi_{\infty}(x_0) =  0.
\end{cases}
$$
Set $\phi(x)= \tilde\phi_{\infty}(x+x_0)$; then  $\phi(x)$ satisfies
$$
\begin{cases}
\displaystyle  d \int_{-\infty}^{0 }  J(x-y)   \phi (y)dy     - d  \phi + c \phi' + f( \phi) =0,  &   x<  0, \\
\displaystyle  \ \ \ \ \  \   \phi (0) =  0.
\end{cases}
$$
Since $ \tilde \phi_{\infty}(x)\in [0,1]$ and is monotone in $x$ with $\tilde\phi_\infty(0)=1/2$, it follows that  $\phi(x)$ is nonincreasing in $x<0$ and $\lim_{x\to-\infty}\phi(x)\in [1/2, 1]$.
The above equation and the property of $f$ then imply that $\phi(-\infty)=1$.
\end{proof}

\subsection{Uniqueness of semi-wave and its monotonicity}
Fix  $c\in (0, c_*)$, and suppose that $\phi_i$, $i=1,2$, are nonnegative solutions of  (\ref{semi-wave})  with speed $c$. Then
$$
\begin{cases}
\displaystyle  d \int_{-\infty}^{0 }  J(x-y)   \phi_i (y)dy     - d  \phi_i + c \phi'_{i} + f( \phi_i) =0,  &   x<  0, \\
\displaystyle  \phi_i(-\infty) =1, \   \phi_i (0) =  0,
\end{cases}
$$
or equivalently
\begin{equation*}
\begin{cases}
\displaystyle  d ( J* \phi_i) (x)     - d  \phi_i + c \phi'_{i} + f( \phi_i) =0  &   x<  0, \\
\displaystyle  \phi_i(-\infty) =1, \   \phi_i (x) =  0,\  x\geq 0.
\end{cases}
\end{equation*}

To prove the uniqueness, it suffices to  show that $\phi_1\equiv \phi_2$.
We first prove a strong maximum principle for later use. We remark that in the following lemma and the uniqueness proof,  only condition {\bf (J)} for the kernel function $J$ is needed.

\begin{lemma}\label{lm-strongMP}
Assume that $w\in C(\mathbb R)\cap C^1(\mathbb R\setminus\{0\})$ satisfies
$$
\begin{cases}
\displaystyle  d ( J* w) (x)     - d  w + a(x) w' + b(x)w \leq 0,  &   x<  0, \\
\displaystyle    w(x) \geq  0,\  x\geq 0,
\end{cases}
$$
where $d$ is a positive constant, $J$ satisfies {\bf (J)}, and $a, b\in L^\infty_{loc}(\mathbb R)$. If $w(x)\geq 0$ and $w(x)\not\equiv 0$, then $w(x)>0$ for $x<0$.
\end{lemma}
\begin{proof}
Suppose that there exists $x_0<0$ such that $w(x_0) =0$. Then $w'(x_0)=0$ and it follows from the differential-integral inequality satisfied by $w$ that  at $x=x_0$,
$$
\displaystyle  d ( J* w) (x_0)     \leq 0,
$$
which indicates that $w(y )=0$ when $y$ is close to $ x_0$. This implies that $w(x) \equiv 0$ when $x<0$, since $\{ x<0 \ |\ w(x)=0\} $ is now both open and closed.
\end{proof}

We are now ready to show $\phi_1\equiv \phi_2$.
Similar to \cite{WZ2019}, for small $\epsilon>0$, define
$$
K_{\epsilon}  = \left\{  k\geq 1 : \ k\phi_1(x) \geq \phi_2(x)-\epsilon \mbox{ for } x\leq 0  \right\}.
$$
$K_{\epsilon} \neq \emptyset $ since $ \phi_i(-\infty) =1$,   $\phi_i (0) =  0$,   $i=1,2$.
Set
$$
k_{\epsilon} = \inf \ K_{\epsilon}\geq 1.
$$
It is clear that $k_{\epsilon}$ is decreasing in $\epsilon$ and thus we may define
$$
k^* = \lim_{\epsilon \rightarrow 0^+} k_{\epsilon}\in[1,+\infty].
$$
From the equation satisfied by $\phi_i$, $i=1,2$, we deduce
\begin{eqnarray*}\label{pf-phi'}
\phi'_{i}(0^-) &=& \lim_{x\rightarrow 0^-} \frac{\phi_i(x)  - \phi_i (0)}{ x} \cr
  &=&  \lim_{x\rightarrow 0^-} {1\over c x}\left[ -d \int_0^x \int_{-\infty}^{0 }  J(z-y)   \phi_i (y)dy  dz + d \int_0^x \phi_i(z)dz - \int_0^x f(\phi_i (z) ) dz \right] \cr
&=& -{d\over c}  \int_{-\infty}^{0 }  J(0-y)   \phi_i (y)dy=\lim_{x\to 0^-}\phi_i'(x)<0.
\end{eqnarray*}
This implies that $k^* < +\infty$. We also have
$$
k^* \phi_1(x ) \geq \phi_2(x),\  \  x\leq 0.
$$
We claim that $k^* =1$. Otherwise, suppose that $k^* >1$  and thus  for $\epsilon>0$ small,  $k_{\epsilon}>1$.
Since $k_{\epsilon} \phi_1(0)  - \phi_2(0) + \epsilon  = \epsilon >0$ and
$$
\lim_{x\rightarrow -\infty}k_{\epsilon} \phi_1(x)  - \phi_2(x) + \epsilon = k_{\epsilon} -1 +\epsilon >0 ,
$$
by the definition of  $k_{\epsilon}$, there exists $x_{\epsilon} \in (-\infty, 0)$ such that
\begin{equation}\label{pf-k-epsilon}
 k_{\epsilon} \phi_1(x_{\epsilon})  - \phi_2(x_{\epsilon}) + \epsilon  =0.
\end{equation}

Now there are three possible cases:
\begin{itemize}
\item{Case (i):} $x_{\epsilon_n} \rightarrow -\infty$ along some sequence $\epsilon_n\rightarrow 0^+$.
\item{Case (ii):}  $x_{\epsilon_n} \rightarrow 0$ along some sequence $\epsilon_n\rightarrow 0^+$.
\item{Case (iii):} $x_{\epsilon_n} \rightarrow x^* \in (-\infty, 0) $ along some sequence $\epsilon_n\rightarrow 0^+$.
\end{itemize}

In {Case (i)}, from (\ref{pf-k-epsilon}) we obtain
$$
0 = \lim_{\epsilon_n \rightarrow 0^+ }   \left( k_{\epsilon_n} \phi_1(x_{\epsilon_n})  - \phi_2(x_{\epsilon_n}) + \epsilon_n  \right) =   k^* -1  >0,
$$
which is impossible.  Hence  Case (i)  leads to a contradiction.

Next,  we consider Cases (ii) and (iii). Define
$$
w_{\epsilon}(x)  =k_{\epsilon} \phi_1(x)  - \phi_2(x) + \epsilon,\; w^*(x) = k^* \phi_1(x) -\phi_2(x).
$$
Then
\[\mbox{
$w_{\epsilon}  (x_{\epsilon}) =0$, $w_{\epsilon} (x) \geq 0$,  $w^*(x) \geq 0$ for $x\in (-\infty, 0]$.}
\]
 Moreover,  $w^*$ satisfies
$$
d ( J*  w^*) (x)     - d  w^*    + c  w^*_{x} +  k^*  f( \phi_1) - f(\phi_2) =0.
$$

In Case (ii),  it follows from \eqref{pf-k-epsilon} that $w^*(0)=0$. Moreover, at $x =  x_{\epsilon_n}$, $w_{\epsilon_n}'( x_{\epsilon_n})=0$, i.e.,  $ k_{\epsilon_n}\phi_1'( x_{\epsilon_n}) = \phi_2'(x_{\epsilon_n})$. Then by leting  $\epsilon_n\rightarrow 0^+$, one has $k^*\phi'_1(0^-) = \phi'_2(0^-)$, and so $w^*_x(0^-)=0$.

On the other hand, from the equation satisfied by $w^*$ and the assumption \textbf{(f3)}, one sees that for $x<0$,
\begin{eqnarray*}
0 &=& d ( J*  w^*) (x)     - d  w^*    + c  w^*_{x} +  k^*  f( \phi_1) - f(\phi_2)\\
 &\geq & d ( J*  w^*) (x)     - d  w^*    + c  w^*_{x} +   f(  k^*\phi_1) - f(\phi_2)\\
 &=&  d ( J*  w^*) (x)     - d  w^*    + c  w^*_{x} +  b(x)w^*,
 \end{eqnarray*}
where the assumption  $k^*>1$ is used, and
\[
b(x):=\left\{\begin{array}{ll}  \frac{f(  k^*\phi_1) - f(\phi_2)}{k^*\phi_1 - \phi_2},& \mbox{ if } \ k^*\phi_1 - \phi_2\not=0,\\
0, &\mbox{ otherwise.}
\end{array}\right.
 \]
 We thus obtain, by letting $x\to 0^-$,
\begin{equation}\label{pf-case2}
d ( J*  w^*) (0) \leq 0.
\end{equation}
Since $k^*>1$, obviously $w^* \not\equiv 0$. Then by Lemma \ref{lm-strongMP}, $w^*>0$ for $x<0$. This is a contradiction to (\ref{pf-case2}). Therefore, Case (ii) also leads to a contradiction.

In Case (iii), similar to the arguments in Case (ii), a contradiction can be derived  at $x= x^*$. Since every possible case leads to a contradiction, we conclude that
$k^*=1$ must happen, which means $\phi_1\geq \phi_2$. Similarly,  we can show $\phi_2\geq \phi_1$. Therefore $\phi_1\equiv \phi_2$ and the uniqueness of the semi-wave (if exists) is verified.

We will from now on assume additoinally {\bf (J2)} is satisfied and denote the unique solution of \eqref{semi-wave} by $\phi^c(x)$.
We are ready to consider the monotonicity of $\phi^c(x)$ in $x$ and in $c$, respectively.

If $\delta>0$, then by Theorem \ref{thm-semiwave-existence}, $w(x)=\phi^c(x-\delta)-\phi^c(x)\geq 0$. Applying Lemma \ref{lm-strongMP} to $\phi^c$ we see that $\phi^c(x)>0$ for $x<0$.
It follows that  $w(x)\not\equiv 0$. We may now apply Lemma  \ref{lm-strongMP} to $w$ to conclude that $w(x)>0$ for $x<0$.
This proves the strict monotonicity of $\phi^c(x)$ in $x$ for $x\in (-\infty, 0]$.

We show next that  $\phi^{c_1} (x)> \phi^{c_2}(x)$ for $x<0$ if $0<c_1< c_2< c_*$.  By Lemma \ref{mono-c} and the proof of Theorem \ref{thm-semiwave-existence}, we see that for such $c_1$ and $c_2$, $w(x):=\phi^{c_1}(x)-\phi^{c_2}(x)$ is nonnegative. Moreover, $(\phi^{c_1})_x\leq 0$ for  $x<0$. Thus $\phi^{c_1}$ satisfies
\begin{equation*}
\begin{cases}
\displaystyle  d ( J* \phi^{c_1}) (x)     - d  \phi^{c_1} + c_2 (\phi^{c_1})_{ x} + f( \phi^{c_1}) \leq 0, &   x<  0, \\
\displaystyle  \phi^{c_1}(-\infty) =1, \   \phi^{c_1} (x) =  0,\  x\geq 0.
\end{cases}
\end{equation*}
We may now apply Lemma  \ref{lm-strongMP} to $w(x)=\phi^{c_1}(x)-\phi^{c_2}(x)$ to conclude that either
$w(x)>0$ for $x<0$ or $w(x)\equiv 0$. If the latter happens, then the above inequality for $\phi^{c_1}$ becomes an equality, which implies that $\phi_x^{c_1}\equiv 0$. But this is a contradiction to $\phi^{c_1}(-\infty)=1>\phi^{c_1}(0)=0$. Therefore $w(x)>0$ for $x<0$ and the strict monotonicity of $\phi^c(x)$ in $c$ is proved.

\medskip

Summarizing, we have proved the following result.

\begin{theorem}\label{thm-semiwave-uniqueness} Suppose that {\bf (J)}, {\bf (J2)} and {\bf (f3)} hold. Then
for any $c\in (0, c_*)$,  the problem \eqref{semi-wave} has a unique solution  $\phi=\phi^c$, and $\phi^c(x)$ is strictly decreasing in $c\in (0, c_*)$ for fixed $x<0$, and is strictly decreasing in $x\in (-\infty, 0]$ for fixed $c\in (0, c_*)$.
\end{theorem}

We conclude this section with the following theorem, which uniquely determines the spreading speed $c_0$.
\begin{theorem}\label{thm-speed} Suppose that {\bf (J)}, {\bf (J2)} and {\bf (f3)} hold. Then
the unique semi-wave $\phi^c(x)$ satisfies
\begin{equation}\label{to-c*}
\lim_{c\to c_*^-}\phi^c(x)=0 \mbox{ locally uniformly in } x\in (-\infty, 0].
\end{equation}
Moreover, for any $\mu>0$, there exists a unique $c=c_0=c_0(\mu)\in (0, c_*)$ such that
\begin{equation}\label{c-mu}
 c_0 = \mu \int_{-\infty}^0 \int_0^{\infty} J(x-y) \phi^{c_0}(x) dy dx .
\end{equation}
\end{theorem}
\begin{proof}
Let $c_n$ be an arbitrary increasing sequence in $(0, c_*)$ converging to $c_*$ as $n\to\infty$. Denote $\phi_n(x):=\phi^{c_n}(x)$.
Then $\phi_n(x)$ is uniformly bounded, and from the equation satisfied by $\phi_n$ we see that $\phi_n'(x)$ is also uniformly bounded. Therefore we can find a subsequence of $\phi_n$, still denoted by itself, such that $\phi_n(x)\to \phi(x)$ in $C_{loc}((-\infty, 0])$ as $n\to\infty$. As in the proof of Theorem \ref{thm-semiwave-existence}, we can verify that $\phi$ satisfies
$$
\begin{cases}
\displaystyle  d \int_{-\infty}^{0 }  J(x-y)   \phi (y)dy     - d  \phi + c_* \phi_{ x} + f( \phi) =0,  &   x<  0, \\
\displaystyle  \ \ \ \ \  \   \phi (0) =  0.
\end{cases}
$$
Clearly we also have $0\leq \phi(x)<\phi_n(x)$ for $x< 0$. We extend $\phi(x)$ by 0 for $x>0$.

Fix $k\in(0,1)$ and define $\tilde\phi(x):=k\phi(x)$. Then by {\bf (f3)} we obtain $f(k\phi)\leq k f(\phi)$ and hence
$$
\begin{cases}
\displaystyle  d (J* \tilde\phi)(x)     - d  \tilde\phi(x) + c_* \tilde \phi'(x)+ f( \tilde\phi(x)) \geq 0,  &   x<  0, \\
\displaystyle  \ \ \ \ \  \   \tilde \phi (x) =  0, & x\geq 0.
\end{cases}
$$
For any $\eta>0$, note that
\[
\phi_*(x-\eta)\geq \phi_*(-\eta) \mbox{ for } x\leq 0.
\]
Since $\phi_*(-\infty)=1$ and $\tilde\phi(x)\leq k<1$, we find that for all large $\eta>0$,
\[
w^\eta(x):=\phi_*(x-\eta)-\tilde\phi(x)\geq 0 \mbox{ for } x\leq 0.
\]
Hence we can define
\[
\eta_*:=\inf\big\{\xi \in\mathbb R:  w^\eta(x)\geq 0 \mbox{ for } x\leq 0 \mbox{ and all } \eta\geq \xi\big\}.
\]
If $\eta_*=-\infty$, then $\tilde\phi(x)\leq \phi_*(x-\eta)$ for all $\eta\in\mathbb R$. Letting $\eta\to -\infty$ and recalling $\phi_*(+\infty)=0$ we immediately obtain $\tilde\phi(x)\leq 0$, which implies $\phi(x)\equiv 0$.

If $\eta_*>-\infty$, then
\[
w^{\eta_*}(x)\geq 0 \mbox{ for } x\leq 0,
\]
and since $w^{\eta_*}(-\infty)\geq 1-k>0$ and $w^{\eta_*}(0)=\phi_*(-\eta_*)>0$, the definition of $\eta_*$ indicates that there exists $x_*\in (-\infty, 0)$ such that
\[
w^{\eta_*}(x_*)=0.
\]
From
\[
 d \int_{-\infty}^{+\infty }  J(x-y)   \phi_* (y-\eta_*)dy     - d  \phi_*(x-\eta_*) + c_*  (\phi_*)'(x-\eta_*) + f( \phi_*(x-\eta_*)) = 0 \mbox{ for } x\in \mathbb R,
 \]
 we obtain
\[
 \begin{cases}
 d (J*\phi_* )(x-\eta_*)     - d  \phi_*(x-\eta_*) + c_*  (\phi_*)' (x-\eta_*)+ f( \phi_*(x-\eta_*)) = 0, & x< 0,\\
 \phi_*(x-\eta_*)>0, & x\geq 0.
 \end{cases}
 \]
 Therefore we can apply Lemma \ref{lm-strongMP} to $w^{\eta_*}$ to conclude that $w^{\eta_*}(x)>0$ for $x<0$, which is a contradiction to $w^{\eta_*}(x_*)=0$. Therefore $\eta_*>-\infty$ cannot occur and we always have $\phi(x)\equiv 0$.
 Since $c_n$ is an arbitrary increasing sequence converging to $c_*$, this implies that \eqref{to-c*} holds.

 It remains to prove \eqref{c-mu}. For $c\in (0, c_*)$ define
 \[
 M(c):=\mu \int_{-\infty}^0 \int_0^{\infty} J(x-y) \phi^{c}(x) dy dx.
 \]
 The monotonicity of $\phi^c(x)$ in $c$ indicates that $M(c)$ is strictly decreasing in $c$.
 Due to the uniqueness of $\phi^c$, one may use a similar argument to that used to show the convergence of $\phi_n(x)$ above to deduce that $\phi^c(x)$ is continuous in $c$ uniformly for $x$ in any bounded set of $(-\infty, 0]$. It follows that $M(c)$ is continuous in $c$. Now we consider the function $c\mapsto c-M(c)$ for $c\in (0, c_*)$.
 Clearly it is continuous and is strictly increasing. By \eqref{to-c*} and the dominated convergence theorem, we see that
 as $c\to c_*^-$, $c-M(c)\to c_*>0$. For all small $c>0$, $c-M(c)\leq c-M(c_*/2)<0$. Therefore there exists a unique $c=c_0\in (0, c_*)$ such that $
 c_0-M(c_0)=0$, i.e., \eqref{c-mu} holds.
 \end{proof}

 \subsection{Semi-wave and condition {\bf (J1)}} In the previous subsection we have proved that when $f$ satisfies {\bf (f3)} and $J$ satisfies {\bf (J)} and {\bf (J2)}, then \eqref{semi-wave}-\eqref{fby} has a unique solution pair $(c,\phi)=(c_0, \phi^{c_0})$, with $\phi^{c_0}(x)$ decreasing in $x$. Now we show that under  condition {\bf (J)}, such a pair $(c_0, \phi^{c_0})$ exists if and only if  {\bf (J1)} holds.

 \begin{theorem}\label{thm-semi-wave}
 Suppose that $J$ satisfies {\bf (J)} and $f$ satisfies {\bf (f3)}. Then \eqref{semi-wave}-\eqref{fby} has a solution pair $(c,\phi)=(c_0, \phi^{c_0})$ with $\phi^{c_0}(x)$ nonincreasing in $x$  if and only if {\bf (J1)} holds. Moreover, when {\bf (J1)} holds, such a pair is unique and $c_0>0$,  $\phi^{c_0}(x)$ is strictly decreasing in $x$.
 \end{theorem}

 We prove Theorem \ref{thm-semi-wave} (which is a restatement of Theorem \ref{thm2}) by two lemmas.

  \begin{lemma}\label{J1-suf}
  Suppose that {\bf (J)} and {\bf (J1)} hold. Then \eqref{semi-wave}-\eqref{fby} has a unique solution pair $(c,\phi)=(c_0, \phi^{c_0})$, and $c_0>0$, $\phi^{c_0}(x)$ is strictly decreasing in $x$.
  \end{lemma}
  \begin{proof} We only need to consider the case that {\bf (J2)} does not hold.
  Let $J_n(x)$ be a sequence satisfying {\bf (J)} and {\bf (J2)} such that
  \[
 \lim_{n\to\infty} J_n(x)=J(x) \mbox{ locally uniformly in } \mathbb R\]
 and
 \[ \lim_{n\to\infty}
 \int_{-\infty}^0\int_0^\infty |J_n(x-y)-J(x-y)|dydx=0.
 \]
 Such a sequence can be easily obtained by letting $J_n(x)=J(x)\xi_n(x)$ with $\xi_n(x)$ a suitable sequence of smooth cut-off functions.

 Since {\bf (J2)} is satisfied by $J_n$,  from Proposition 2.1 we obtain a minimal wave speed $c_*=c_*^n>0$. We must have $\lim_{n\to\infty}c_*^n=+\infty$, for otherwise by passing to a subsequence we may assume $\lim_{n\to\infty} c_*^n=c_*^\infty\in [0, +\infty)$, and then by a similar argument \footnote{The case $c_*^\infty=0$ has to be proved differently, as in the last part of the proof of this lemma.
} to the proof of
 Theorem \ref{thm-semiwave-existence}, we can show that the Fisher-KPP equation in Proposition \ref{tw} with kernel function $J$ satisfying {\bf (J)} but not {\bf (J2)}
 has a traveling wave with speed $c_*^\infty$, which is a contradiction to the second part of the conclusion in that proposition.  Therefore, for any fixed $c>0$ and all
 large $n$, we have $0<c<c_*^n$ and so \eqref{semi-wave} with $J$ replaced by $J_n$ has a unique solution $\phi=\phi_n^c$. Moreover, we may argue as in the proof of Theorem \ref{thm-semiwave-existence} to conclude that
  $\phi^{c}_n(x)\to  \hat\phi^c(x)$ and for some $x_0>0$,   $\phi^c(x):= \hat \phi(x-x_0)$ satisfies \eqref{semi-wave}. The monotonicity of $\phi^c_n$ in $x$ and in $c$
  then implies that $\phi^c(x)$ is nonincreasing  in $x\in (-\infty, 0]$ for fixed $c>0$, and nonincreasing in $c\in (0, +\infty)$ for fixed $x<0$. We may now use Lemma \ref{lm-strongMP} as in the proof of Theorem \ref{thm-semiwave-uniqueness}, where the property {\bf (J2)} is not needed, to conclude that all the conclusions for $\phi^c$ in Theorem \ref{thm-semiwave-uniqueness} still hold for the current $\phi^c$.

 By Theorem \ref{thm-speed}, for each $n$, \eqref{semi-wave}-\eqref{fby} with $J$ replaced by $J_n$ has a unique solution pair
  $(c_{0,n}, \phi^{c_{0,n}})$. By \eqref{hat-c}, we have
  \[
  c_{0,n}\in (0, \mu c(J_n)) \ \ \mbox{ with  $ c(J_n)\to c(J)>0$ as $n\to\infty$.}
  \]
  Therefore, by passing to a subsequence we may assume that
  \[
  c_{0,n}\to c_0\in [0, \mu c(J)] \mbox{  as } n\to\infty.
  \]

  If $c_0>0$, then we may argue as in the proof of Theorem \ref{thm-semiwave-existence} to conclude that
  $\phi^{c_{0,n}}(x)\to \hat \phi(x)$ and for some $x_0>0$, $(c,\phi(x))=(c_0, \hat \phi(x-x_0))$ satisfies \eqref{semi-wave}-\eqref{fby}. As the conclusions in Theorem \ref{thm-semiwave-uniqueness} still hold we necessarily have $\hat\phi(x-x_0)=\phi^{c_0}(x)$. Moreover, $c_0>0$ is the unique $c$ such that \eqref{fby} holds.

  If $c_0=0$, we show that a contradiction occurs. For convenience, we denote
  \[
  \phi_n=\phi^{c_0,n},\; c_n=c_{0,n}.
  \]
  Then $\lim_{n\to\infty} c_n=0$ and
  $$
\begin{cases}
\displaystyle  d \int_{-\infty}^{0 }  J_n(x-y)   \phi_n (y)dy     - d  \phi_n + c_n \phi_n'+ f( \phi_n) =0,\; \phi_n'<0,   & \mbox{ for }  x<  0, \\
\displaystyle  \ \ \ \ \  \   \phi_n (0) =  0, \; \phi_n(-\infty)=1,& \\
\displaystyle \mu\int_{-\infty}^0\int_0^\infty J_n(x-y)\phi_n(x)dydx=c_n.&
\end{cases}
$$
Using $c_n\to 0$  we easily see that
\[
\lim_{n\to\infty}\int_{-\infty}^0\int_0^\infty J(x-y)\phi_n(x)dydx=0,
\]
and hence, due to the monotonicity of each $\phi_n$, and the assumption that $J$ does not satisfy {\bf (J2)} (and so it does not have compact support), we obtain
\[
\lim_{n\to\infty} \phi_n(x)=0 \mbox{ uniformly on every bounded interval in } (-\infty, 0].
\]
Choose $x_n<0$ such that $\phi_n(x_n)=1/2$. Then $x_n\to-\infty$ as $n\to\infty$. We now define
\[
\tilde\phi_n(x):=\phi_n(x+x_n).
\]
Then
$$
\begin{cases}
\displaystyle  d \int_{-\infty}^{-x_n }  J_n(x-y)   \tilde\phi_n (y)dy     - d  \tilde\phi_n + c_n \tilde\phi_n'+ f(\tilde \phi_n) =0,\; \tilde\phi_n'<0,   & \mbox{ for }  x<  -x_n, \\
\displaystyle  \ \ \ \ \  \   \tilde\phi_n (0) =  1/2, \; \tilde\phi_n(-\infty)=1.&
\end{cases}
$$
Since $\tilde\phi_n'<0$ and $0\leq \tilde\phi_n\leq 1$, by Helly's theorem, $\{\tilde\phi_n\}$ has a subsequence, which for convenience we still denote by itself, such that, as $n\to\infty$,
$\tilde\phi_n(x)\to \tilde\phi(x)$ for almost every $x\in\mathbb R$. Clearly $\tilde\phi(x)$ is nonincreasing and $\tilde\phi(0)=1/2$.

From the above equations for $\tilde\phi_n$ we obtain, for any $z\in \mathbb R$ and all large $n$,
\[
d\int_0^z\int_{-\infty}^{-x_n}J_n(x-y)\tilde\phi_n(y)dydx-d\int_0^z\tilde\phi_n(x)dx+c_n\tilde\phi_n(z)-c_n/2+\int_0^zf(\tilde\phi_n(x))dx=0.
\]
 Letting $n\to\infty$ and making use of the dominated convergence theorem, we deduce
\[
d\int_0^z\int_{-\infty}^{\infty}J(x-y)\tilde\phi(y)dydx-d\int_0^z\tilde\phi(x)dx+\int_0^zf(\tilde\phi(x))dx=0.
\]
Since $z\in\mathbb R$ is arbitrary, this implies that
\[
d\int_{-\infty}^{\infty}J(x-y)\tilde\phi(y)dy-d\tilde\phi(x)+f(\tilde\phi(x))=0 \mbox{ for a.e. } x\in\mathbb R.
\]
But this implies that $\tilde\phi$ is a traveling wave with speed $c=0$, a contradiction to Proposition \ref{tw}, since we have assumed that $J$ does not satisfy {\bf (J2)}. Thus we have proved that $c_0=0$ cannot happen, and the proof is complete.
\end{proof}

 \begin{lemma}\label{Ji-nec}
 Suppose that {\bf (J)} holds and \eqref{semi-wave}-\eqref{fby} has a solution pair $(c,\phi)=(c_0, \phi^{c_0})$ with $\phi^{c_0}(x)$ nonincreasing in $x$.
 Then $J$ satisfies {\bf (J1)}.
  \end{lemma}
  \begin{proof}
  Since $\phi^{c_0}(x)$ is nonincreasing in $x$, we have
  \[
  c_0=\mu\int_{-\infty}^0\int_0^\infty J(x-y)\phi^{c_0}(x)dydx\geq \mu \,\phi^{c_0}(-1)\int_{-\infty}^{-1}\int_0^\infty J(x-y)dydx.
  \]
  Thus
  \[
  \int_{-\infty}^{-1}\int_0^\infty J(x-y)dydx=\int_{-\infty}^{-1}a(x)dx<+\infty.
  \]
  Since $a(x)$ is continuous, clearly
  \[
  \int_{-\infty}^{0}\int_0^\infty J(x-y)dydx=\int_{-\infty}^{-1}a(x)dx+\int_{-1}^0 a(x)dx<+\infty.
  \]
  Hence {\bf (J1)} holds.
  \end{proof}

\section{Spreading speed of \eqref{main-fb}}

Suppose that $f$ satisfies {\bf (f3)} and $J$ satisfies {\bf (J)} and {\bf (J1)}. Then there exists a unique pair $(c_0, \phi^{c_0})$ satisfying \eqref{semi-wave} and \eqref{fby}. Let $(u,g,h)$ be the unique solution of \eqref{main-fb} and suppose that spreading happens, that is
\[
\lim_{t\to\infty}h(t)=-\lim_{t\to\infty} g(t)=\infty,\; \mbox{ and } \lim_{t\to\infty} u(t,x)=1 \mbox{ locally uniformly in $x\in\mathbb R$}.
\]
We are going to show that
\[
\lim_{t\to\infty}\frac{h(t)}{t}=-\lim_{t\to\infty}\frac{g(t)}{t}=c_0,
\]
which is part (i) of Theorem \ref{thm1}.

It suffices to show the conclusion for $h(t)$, as $\tilde u(t,x):=u(t,-x)$ satisfies \eqref{main-fb} with free boundaries
$x=\tilde h(t):=-g(t),\; x=\tilde g(t):=-h(t)$ and initial function $\tilde u_0(x):=u_0(-x)$.

\begin{lemma}\label{<c0} Under the above assumptions, we have
\[
\limsup_{t\to\infty}\frac{h(t)}{t}\leq c_0.
\]
\end{lemma}
\begin{proof}
For any given $\epsilon>0$  we define
\[
\delta:=2\epsilon c_0,\; \overline h(t):=(c_0+\delta)t+L,\; \overline u(t,x)=: (1+\epsilon)\phi^{c_0}(x-\overline h(t)),
\]
with $L>0$ to be determined. A simple comparison argument with the ODE problem
\[
v'=f(v),\; v(0)=\|u_0\|_\infty
\]
shows that $u(t,x)\leq v(t)$ and hence
\[
\limsup_{t\to\infty}u(t,x)\leq 1 \mbox{ uniformly for $x\in [g(t), h(t)]$}.
\]
Thus there exists $T>0$ large so that
\[
u(T+t,x)\leq 1+\frac\epsilon 2 \ \mbox{ for }\ t\geq 0, \ \ x\in [g(T+t), h(T+t)].
\]
Since $\phi^{c_0}(-\infty)=1$, we may choose $L>0$ large such that $\overline h(0)=L>h(T)$ and
\begin{equation}\label{u0}
\overline u(0,x)=(1+\epsilon)\phi^{c_0}(x-L)>1+\frac \epsilon 2\geq u(T,x) \mbox{ for } x\in [g(T), h(T)],
\end{equation}

We show next that
\begin{equation}\label{diff-ineq}
\overline u_t\geq d\int_{g(t+T)}^{\overline h(t)}J(x-y)\overline u(t,y)dy-d\overline u(t,x)+f(\overline u(t,x))
\end{equation}
for $t>0$ and $x\in [g(t+T), \overline h(t)]$, and
\begin{equation}\label{h'}
\overline h'(t)> \mu \int_{g(t+T)}^{\overline h(t)}\int_{\overline h(t)}^\infty J(x-y)\overline u(t,x)dydx \mbox{ for } t>0.
\end{equation}
Indeed,
\begin{eqnarray*}
\overline u_t &=& -(1+\epsilon)(c_0+\delta)(\phi^{c_0})'(x-\overline h(t))>-(1+\epsilon) c_0(\phi^{c_0})'(x-\overline h(t))\\
& = & (1+\epsilon)\left[d\int_{-\infty}^{\overline h(t)}J(x-y)\phi^{c_0}(y-\overline h(t))dy-d\phi^{c_0}(x-\overline h(t))+f(\phi^{c_0}(x-\overline h(t))\right]\\
&=&d\int_{-\infty}^{\overline h(t)}J(x-y)\overline u(t,y)dy-d\overline u(t,x)+(1+\epsilon)f(\phi^{c_0}(x-\overline h(t))\\
&\geq&d\int_{-\infty}^{\overline h(t)}J(x-y)\overline u(t,y)dy-d\overline u(t,x)+f(\overline u(t,x))\\
&\geq&d\int_{g(t+T)}^{\overline h(t)}J(x-y)\overline u(t,y)dy-d\overline u(t,x)+f(\overline u(t,x))
 \end{eqnarray*}
 for $t>0$ and $x<\overline h(t)$, where we have used {\bf (f3)}. This proves \eqref{diff-ineq}.

 To show \eqref{h'} we calculate
 \begin{eqnarray*}
 &&\mu \int_{g(t+T)}^{\overline h(t)}\int_{\overline h(t)}^\infty J(x-y)\overline u(t,x)dydx\\
 &\leq &\mu \int_{-\infty}^{\overline h(t)}\int_{\overline h(t)}^\infty J(x-y)\overline u(t,x)dydx\\
  &= &\mu (1+\epsilon)\int_{-\infty}^{0}\int_{0}^\infty J(x-y)\phi^{c_0}(x)dydx\\
 &=& (1+\epsilon)c_0<c_0+\delta=\overline h'(t).
 \end{eqnarray*}
  Thus \eqref{h'} holds.

 We are now ready to show that
 \begin{equation}\label{u<u}
 h(t+T)< \overline h(t) \mbox{ and } u(t+T,x)< \overline u(t,x) \mbox{ for } t>0,\; x\in [g(t+T), h(t+T)].
 \end{equation}
 By \eqref{u0} and $\overline h(0)>h(T)$, we see that the above inequalities hold for $t>0$ small. If the above inequalities do not hold for all $t>0$, then there is a first time moment $t^*>0$ such that at least one of them is violated at $t=t^*$, i.e.,
 the above inequalities hold for $t\in (0, t^*)$, and
 \[
 \hspace{-10cm}{\rm (i)}\ \ \   h(t^*+T)= \overline h(t^*),\ \ \mbox{ or } \vspace{-0.3cm}
 \]
 \[
{\rm (ii)}\ \  h(t^*+T)<\overline h(t^*) \mbox{ and } u(t^*+T,x^*)= \overline u(t^*,x^*) \mbox{ for some } x^*\in [g(t^*+T), h(t^*+T)].
 \]
 If (i) happens, then necessarily $h'(t^*+T)\geq \overline h'(t^*)$. On the other hand,
 \begin{eqnarray*}
 \overline h'(t^*)&>&\mu \int_{g(t^*+T)}^{\overline h(t^*)}\int_{\overline h(t^*)}^\infty J(x-y)\overline u(t^*,x)dydx\\
 &=&\mu\int_{g(t^*+T)}^{h(t^*+T)}\int_{h(t^*+T)}^\infty J(x-y) \overline u(t^*, x)dydx\\
 &\geq &\mu \int_{g(t^*+T)}^{h(t^*+T)}\int_{h(t^*+T)}^\infty J(x-y) u(t^*+T,x)dydx\\
 & =&h'(t^*+T),
 \end{eqnarray*}
 where we have used
 \[
 u(t^*+T,x)\leq \overline u(t^*,x) \mbox{ for }  x\in [g(t^*+T), h(t^*+T)].
 \]
 Thus (i) leads to a contradiction.

 If (ii) happens, then due to $\overline u(t,x)>0$ for $x\in \{g(t+T), h(t+T)\}$ for $t\in (0, t^*]$, and $\overline u(0,x)>u(T,x)$ for $x\in [g(T), h(T)]$, we can use the comparison principle in \cite{CDLL2019} to conclude that
 \[
 \overline u(t^*,x)>u(t^*+T,x) \mbox{ for } x\in [g(t+T), h(t+T)],
 \]
 and so we again reach a contradiction. Therefore \eqref{u<u} holds, and
 \[
 \limsup_{t\to\infty}\frac{h(t)}{t}\leq \lim_{t\to\infty} \frac{\overline h(t-T)}{t}=c_0+\delta=c_0+2\epsilon c_0.
 \]
 Letting $\epsilon\to 0$, we immediately obtain $\limsup_{t\to\infty} h(t)/t\leq c_0$.
 \end{proof}

\begin{lemma}\label{>c0} Under the assumptions of Lemma \ref{<c0}, we have
\[
\liminf_{t\to\infty}\frac{h(t)}{t}\geq c_0.
\]
\end{lemma}
\begin{proof} Since $f'(1)<0$, there exists $\delta_0>0$ small such that $f'(u)<0$ for $u\in [1-\delta_0, 1]$.
For any given $\epsilon\in (0, \delta_0]$, we define
\[
\delta:=2\epsilon c_0,\;\ \  \underline h(t):=(c_0-\delta)t+L\;\;\; \mbox{ and }
\]
\[
\underline u(t,x):=(1-\epsilon)\left[\phi^{c_0}(x-\underline h(t))+\phi^{c_0}(-x-\underline h(t))-1\right],
\]
with $L>0$ a large constant to be determined.
Clearly, for $t\geq 0$,
\[
0>\underline u(t,\pm \underline h(t))=(1-\epsilon)\left[\phi^{c_0}(-2\underline h(t))-1\right]\geq (1-\epsilon)\left[\phi^{c_0}(-2L)-1\right]
\to 0 \mbox{ as } L\to\infty.
\]
We show next that, if $L$ is chosen large enough, then
\begin{equation}\label{h'2}
\underline h'(t)\leq \mu\int_{-\underline h(t)}^{\underline h(t)}\int_{\underline h(t)}^\infty J(x-y)\underline u(t,x)dydx
\;\mbox{ for } t>0,
\end{equation}
and
\begin{equation}\label{u_t}
\underline u_t\leq d\int_{-\underline h(t)}^{\underline h(t)} J(x-y)\underline u(t,y)dy-d\underline u+f(\underline u)
\mbox{ for } x\in (-\underline h(t), \underline h(t)),\; t>0.
\end{equation}

To show \eqref{h'2}, we calculate
\begin{eqnarray*}
 &&\mu \int_{-\underline h(t)}^{\underline h(t)}\int_{\underline h(t)}^\infty J(x-y)\underline u(t,x)dydx\\
 &= &\mu(1-\epsilon) \int_{-2\underline h(t)}^{0}\int_{0}^\infty J(x-y)\phi^{c_0} (x)dydx\\
 &&+\mu(1-\epsilon)\int_{-2\underline h(t)}^{0}\int_{0}^\infty J(x-y)\left[\phi^{c_0}(-x-2\underline h(t))-1\right]dydx\\
  &= &(1-\epsilon)c_0-\mu(1-\epsilon)\int_{-\infty}^{-2\underline h(t)}\int_0^\infty J(x-y)\phi^{c_0} (x)dydx\\
 &&-\mu(1-\epsilon)\int_{-2\underline h(t)}^{0}\int_{0}^\infty J(x-y)\left[1-\phi^{c_0}(-x-2\underline h(t))\right]dydx.
 \end{eqnarray*}
 By {\bf (J1)}, we have, for all $t\geq 0$,
 \begin{eqnarray*}
0&\leq &\mu(1-\epsilon)\int_{-\infty}^{-2\underline h(t)}\int_0^\infty J(x-y)\phi^{c_0} (x)dydx\\
&\leq & \mu(1-\epsilon)\int_{-\infty}^{-2L}\int_0^\infty J(x-y)dydx<\frac 14 \epsilon c_0
\end{eqnarray*}
 provided that $L$ is large enough, say $L\geq L_1$. Moreover,
 \begin{eqnarray*}
0&\leq &\mu(1-\epsilon)\int_{-2\underline h(t)}^{0}\int_{0}^\infty J(x-y)\left[1-\phi^{c_0}(-x-2\underline h(t))\right]dydx\\
&\leq &\mu(1-\epsilon)\int_{-\infty}^{0}\int_{0}^\infty J(x-y)\left[1-\phi^{c_0}(-x-2\underline h(t))\right]dydx\\
&\leq&\mu(1-\epsilon)\int_{-\infty}^{-2L_1}\int_{0}^\infty J(x-y)dydx+\mu(1-\epsilon)\int_{-2L_1}^{0}\int_{0}^\infty J(x-y)\left[1-\phi^{c_0}(-2L)\right]dydx\\
&<&\frac 14 \epsilon c_0+\mu(1-\epsilon)\left[1-\phi^{c_0}(-2L)\right]\int_{-2L_1}^{0}\int_{0}^\infty J(x-y)dydx\\
&<&\frac 12\epsilon c_0\;\; \mbox{ for all } t\geq 0,
\end{eqnarray*}
provided that $L\geq L_2$ for some large enough $L_2>L_1$.

Thus for $L\geq L_2$ and all $t\geq 0$,
\[
\mu \int_{-\underline h(t)}^{\underline h(t)}\int_{\underline h(t)}^\infty J(x-y)\underline u(t,x)dydx
>(1-\epsilon)c_0-\frac 34 \epsilon c_0>(1-2\epsilon)c_0=\underline h'(t).
\]
This proves \eqref{h'2}.

Next we prove \eqref{u_t}. Firstly we need to extend $f(u)$ by defining
\[
f(u)=f'(0)u \mbox{ for } u<0.
\]
Secondly we fix several constants for later use. Due to $f(1)=0$ and $f'(u)<0$ for $u\in [1-\epsilon, 1]$,
we can choose $\tilde \epsilon>0$ small enough such that
\begin{equation}\label{tilde-ep}
2(1-\epsilon)f(1-\frac{\tilde \epsilon}2)<f(1-\epsilon) \mbox{ and } f'(u)<0 \mbox{ for } u\in [(1-\epsilon)(1-\tilde\epsilon), 1].
\end{equation}
Then using $\phi^{c_0}(-\infty)=1$ we can find $M>0$ large enough such that
\[
\phi^{c_0}(-M)>1-\frac{\tilde \epsilon}2,
\]
 which implies, in particular,
\begin{equation}\label{M-M}
\phi^{c_0}(x-\underline h(t)),\; \phi^{c_0}(-x-\underline h(t))\in (1-\frac{\tilde\epsilon}2, 1) \mbox{ for } x\in [-\underline h(t)+M,
\underline h(t)-M].
\end{equation}
Define
\[
\epsilon_0:=\inf_{x\in [-M, 0]}|(\phi^{c_0})'(x)|>0;
\]
then clearly
\begin{equation}\label{ep0}
\left\{
\begin{array}{ll}
(\phi^{c_0})'(x-\underline h(t))\leq -\epsilon_0 &\mbox{ for } x\in [\underline h(t)-M, \underline h(t)];\\
(\phi^{c_0})'(-x-\underline h(t))\leq -\epsilon_0 &\mbox{ for } x\in [-\underline h(t),  -\underline h(t)+M].
\end{array}
\right.
\end{equation}
Finally we set
\[
M_0:=\max_{u\in [0,1]}|f'(u)|,\;
\hat\epsilon:=\frac{1-\epsilon}{2M_0}\delta \epsilon_0.
\]

To simplify notations, in the following we write $\phi=\phi^{c_0}$. We have
\begin{eqnarray*}
\underline u_t&=& -(1-\epsilon)(c_0-\delta)\Big[\phi'(x-\underline h(t))+\phi'(-x-\underline h(t))\Big]\\
&=&(1-\epsilon)\delta \Big[\phi'(x-\underline h(t))+\phi'(-x-\underline h(t))\Big]\\
&&+(1-\epsilon)\Big[d\int_{-\infty}^0J(x-\underline h(t)-y)\phi(y)dy-d\phi(x-\underline h(t))+f(\phi(x-\underline h(t))\Big]\\
&&+(1-\epsilon)\Big[d\int_{-\infty}^0J(-x-\underline h(t)-y)\phi(y)dy-d\phi(-x-\underline h(t))+f(\phi(-x-\underline h(t))\Big]\\
&=&(1-\epsilon)\delta \Big[\phi'(x-\underline h(t))+\phi'(-x-\underline h(t))\Big]\\
&&+(1-\epsilon)\Big[d\int_{-\infty}^{\underline h(t)}J(x-y)\phi(y-\underline h(t))dy-d\phi(x-\underline h(t))\\
&& \ \ \ \ \ \ \ \ \ \ \ \ \ \    +
d\int_{-\underline h(t)}^\infty J(-x+y)\phi(-y-\underline h(t))dy-d\phi(-x-\underline h(t))\Big]\\
&&  +(1-\epsilon)\Big[f(\phi(x-\underline h(t))+f(\phi(-x-\underline h(t))\Big]\\
&=& (1-\epsilon)\delta \Big[\phi'(x-\underline h(t))+\phi'(-x-\underline h(t))\Big]\\
&&+d\int_{-\underline h(t)}^{\underline h(t)} J(x-y)\underline u(t,y)dy-d\underline u(t,x)\\
&&+(1-\epsilon)d\Big[\int_{-\infty}^{-\underline h(t)} J(x-y)[\phi(y-\underline h(t))-1]dy+\int_{\underline h(t)}^\infty J(x-y)[\phi(-y-\underline h(t))-1]dy\Big]\\
&&+(1-\epsilon)\Big[f(\phi(x-\underline h(t))+f(\phi(-x-\underline h(t))\Big]\\
&\leq & d\int_{-\underline h(t)}^{\underline h(t)} J(x-y)\underline u(t,y)dy-d\underline u(t,x)\\
&&+(1-\epsilon)\delta \Big[\phi'(x-\underline h(t))+\phi'(-x-\underline h(t))\Big]+(1-\epsilon)\Big[f(\phi(x-\underline h(t))+f(\phi(-x-\underline h(t))\Big]\\
&=&d\int_{-\underline h(t)}^{\underline h(t)} J(x-y)\underline u(t,y)dy-d\underline u(t,x) +f(\underline u(t,x))+ \delta(t,x)
\end{eqnarray*}
with
\begin{eqnarray*}
\delta(t,x):&=&(1-\epsilon)\delta \Big[\phi'(x-\underline h(t))+\phi'(-x-\underline h(t))\Big]\\
&&+(1-\epsilon)\Big[f(\phi(x-\underline h(t))+f(\phi(-x-\underline h(t))\Big]-f(\underline u(t,x)).
\end{eqnarray*}
To prove \eqref{u_t}, it suffices to show that
\begin{equation}\label{J<0}
\delta(t,x)\leq 0 \mbox{ for } x\in [-\underline h(t), \underline h(t)],\; t\geq 0.
\end{equation}

We start by checking the case $x\in [\underline h(t)-M, \underline h(t)] $ and $t\geq 0$. For such $x$ and $t$,  we have
\[
0>\phi(-x-\underline h(t))-1\geq \phi(-2\underline h(t)+M)-1\geq \phi(-2L+M)-1\geq -\hat \epsilon
\]
provided that $L$ is large enough, say $L\geq L_3\geq L_2$. It follows that
\[
f(\underline u(t,x))\geq f\big((1-\epsilon)\phi(x-\underline h(t))\big)-M_0(1-\epsilon)\hat \epsilon,
\]
\[
f\big(\phi(-x-\underline h(t))\big)=f\big(\phi(-x-\underline h(t))\big)-f(1)\leq M_0\hat\epsilon,
\]
and hence, by \eqref{ep0} and {\bf (f3)},
\begin{eqnarray*}
\delta (t,x)&\leq& -(1-\epsilon)\delta \epsilon_0+(1-\epsilon)\big[f\big(\phi(x-\underline h(t))\big)+M_0\hat\epsilon\big] \\
&&-f\big((1-\epsilon)\phi(x-\underline h(t))\big)+M_0(1-\epsilon)\hat\epsilon\\
&\leq&-(1-\epsilon)\delta \epsilon_0+2(1-\epsilon)M_0\hat\epsilon<0.
\end{eqnarray*}
Since $\delta(t,-x)=\delta(t,x)$, the above inequality also holds for $x\in [-\underline h(t), -\underline h(t)+M] $ and $t\geq 0$.

It remains to check the case $x\in [-\underline h(t)+M, \underline h(t)-M] $ and $t\geq 0$. Now \eqref{M-M} holds and so
\[
\underline u(t,x)\in [(1-\epsilon)(1-\tilde\epsilon), 1-\epsilon].
\]
Since $f(u)$ is decreasing for $u\in [(1-\epsilon)(1-\tilde\epsilon), 1-\epsilon]$, it follows that, for such $x$ and $t$,
\[
\delta (t,x)< (1-\epsilon)\big[f(1-\frac{\tilde\epsilon}2)+f(1-\frac{\tilde\epsilon}2)\big]-f(1-\epsilon)<0
\]
due to \eqref{tilde-ep}. Thus \eqref{J<0} holds. This proves \eqref{u_t}.

Since $J(-x)=J(x)$ and $\underline u(t,-x)=\underline u(t,x)$, from \eqref{h'2} we easily deduce
\begin{equation}\label{h'1}
-\underline h'(t)\geq - \mu\int_{-\underline h(t)}^{\underline h(t)}\int_{-\infty}^{-\underline h(t)} J(x-y)\underline u(t,x)dydx
\;\mbox{ for } t>0.
\end{equation}

We are now ready to compare $(u, g, h)$ with $(\underline u, -\underline h, \underline h)$ by a comparison argument.
Since spreading happens for $(u, g, h)$, there exists $T>0$ large enough such that
\[
g(T)<-L=-\underline h(0),\; h(T)>L=\underline h(0),\; u(T,x)>1-\epsilon>\underline u(0,x) \mbox{ for } x\in [-L, L].
\]
In view of \eqref{h'2},   \eqref{u_t} and \eqref{h'1}, we may now use the lower solution version of Theorem 3.1 in \cite{CDLL2019} to deduce
\[
g(T+t)\leq -\underline h(t),\; h(t+T)\geq \underline h(t) \mbox{ and } u(t+T, x)\geq \underline u(t,x) \mbox{ for } t>0,\; x\in
[-\underline h(t), \underline h(t)].
\]
In particular,
\[
\liminf_{t\to\infty}\frac{h(t)}{t}\geq \lim_{t\to\infty}\frac{\underline h(t-T)}{t}=c_0-\delta=(1-2\epsilon)c_0.
\]
Letting $\epsilon\to 0$ we obtain $\liminf_{t\to\infty}\frac{h(t)}{t}\geq c_0$. This completes the proof.
\end{proof}

\section{Accelerating spreading of \eqref{main-fb}}

In this section, we prove part (ii) of Theorem \ref{thm1}. So throughout this section, we assume that {\bf (f3)} and {\bf (J)} hold, but {\bf (J1)} is not satisfied. Moreover, we assume that  $(u,g,h)$ is the unique solution of \eqref{main-fb}, and spreading happens, namely
\[
\lim_{t\to\infty} h(t)=-\lim_{t\to\infty} g(t)=\infty,\; \lim_{t\to\infty} u(t,x)=1 \mbox{ locally uniformly for } x\in\mathbb R.
\]

We firstly choose a sequence $\{J_n(x)\}$ such that each $J_n(x)$ is nonnegative, continuous, even,   has nonempty compact support, and
\[
J_n(x)\leq J_{n+1}(x)\leq J(x) \mbox{ for all } n\geq 1,\; x\in\mathbb R,\;\;
J_n(x)\to J(x) \mbox{ in } L^1(\mathbb R).
\]
Such a sequence can be easily constructed by defining $J_n(x)=J(x)\xi_n(x)$, with $\xi_n(x)$ a suitable sequence of smooth cut-off functions. We then consider the auxiliary problem which is obtained by replacing $J$ by $J_n$ in \eqref{main-fb}, namely
\begin{equation}\label{main-n}
\begin{cases}
\displaystyle u_t=d\int_{g(t)}^{h(t)}J_n(x-y)u(t,y)dy-du(t,x)+f(u),
& t>0,~x\in(g(t),h(t)),\\
\displaystyle u(t,g(t))=u(t,h(t))=0, &t>0,\\
\displaystyle h'(t)=\mu\int_{g(t)}^{h(t)}\int_{h(t)}^{+\infty}
J_n(x-y)u(t,x)dydx, &t>0,\\
\displaystyle g'(t)=-\mu\int_{g(t)}^{h(t)}\int_{-\infty}^{g(t)}
J_n(x-y)u(t,x)dydx, &t>0,\\
\displaystyle u(0,x)=u_0(x),~h(0)=-g(0)=h_0,  &  x\in[-h_0,h_0],
\end{cases}
\end{equation}

\begin{lemma}\label{lem4.1}
For every large $n$, problem \eqref{main-n} has a unique positive solution $(u_n, g_n, h_n)$ defined for all $t>0$. Moreover,
\[
h(t)\geq h_{n+1}(t)\geq h_n(t),\; g(t)\leq g_{n+1}(t)\leq g_n(t) \mbox{ for all } t>0,\; n\geq 1.
\]
\end{lemma}
\begin{proof}
Define $\sigma_n:=\int_{\mathbb R} J_n(x) dx$. Then $\sigma_n\in (0, 1]$, is nondecreasing in $n$, and $\lim_{n\to\infty}\sigma_n=1$.
Set
\[
\tilde J_n(x):=\frac 1{\sigma_n}J_n(x).
\]
Clearly $\tilde J_n$ satisfies {\bf (J)}. Moreover, since $\tilde J_n$ has compact support, it also satisfies {\bf (J1)} (and  {\bf (J2)} as well). Set
\[
f_n(u):=f(u)-(1-\sigma_n)u.
\]
Then at least for all large $n$, $f_n$ satisfies {\bf (f3)} except that $f(1)=0>f'(1)$ should be replaced by
$f_n(\eta_n)=0>f'(\eta_n)$ for some uniquely determined $\eta_n\in (0,1)$, and $\lim_{n\to\infty}\eta_n=1$.

We may now rewrite \eqref{main-n} in the equivalent form
\begin{equation}\label{main-n-tilde}
\begin{cases}
\displaystyle u_t=d\sigma_n\int_{g(t)}^{h(t)}\tilde J_n(x-y)u(t,y)dy-d\sigma_n u(t,x)+f_n(u),
& t>0,~x\in(g(t),h(t)),\\
\displaystyle u(t,g(t))=u(t,h(t))=0, &t>0,\\
\displaystyle h'(t)=\mu\sigma_n \int_{g(t)}^{h(t)}\int_{h(t)}^{+\infty}
\tilde J_n(x-y)u(t,x)dydx, &t>0,\\
\displaystyle g'(t)=-\mu\sigma_n \int_{g(t)}^{h(t)}\int_{-\infty}^{g(t)}
\tilde J_n(x-y)u(t,x)dydx, &t>0,\\
\displaystyle u(0,x)=u_0(x),~h(0)=-g(0)=h_0,  &  x\in[-h_0,h_0].
\end{cases}
\end{equation}
By \cite{CDLL2019}, we know that, for every large $n$, \eqref{main-n-tilde} has a unique solution $(u_n, g_n, h_n)$ which is defined for all $t>0$.

Since $(u_n, g_n, h_n)$ satisfies \eqref{main-n}, it follows from $J_n(x)\leq J(x)$ that $(u_n, g_n, h_n)$ is a lower solution of
\eqref{main-fb}, and hence we can use the comparison principle in \cite{CDLL2019} to conclude that
\[
g_n(t)\geq g(t),\; h(t)\geq h_n(t) \mbox{ for all } t>0.
\]
Similarly from $J_n\leq J_{n+1}$ we deduce
\[
g_n(t)\geq g_{n+1}(t),\; h_{n+1}(t)\geq h_n(t) \mbox{ for all } t>0 \mbox{ and every large } n.
\]
The proof is complete.
\end{proof}
\begin{lemma}\label{lem4.2}
For every large $n$, spreading happens to $(u_n, g_n, h_n)$.
\end{lemma}
\begin{proof}
 For any fixed constant $\ell>0$, let $\mathcal L_{\ell}$ denote the operator defined by
\[
\mathcal L_{\ell}\,[\phi](x):=d\int_{-\ell}^{\ell}J(x-y)\phi(y)dy-d\phi(x),
\]
and let $\lambda_p(\mathcal L_{\ell}+a)$ denote the principal eigenvalue of $\mathcal L_{\ell}+a$ given by
\[
\lambda_p(\mathcal L_{\ell}+a):=\inf\Big\{\lambda\in\mathbb R: \mathcal L_{\ell}\,[\phi]+a\phi \leq \lambda \phi \mbox{ in } (-\ell,\ell)
\mbox{ for some } \phi\in C([-\ell,\ell]),\; \phi>0\Big\}.
\]
By Proposition 3.4 in \cite{CDLL2019} we have
\[
\lim_{\ell\to\infty} \lambda_p(\mathcal L_{\ell}+f'(0))=f'(0)>0.
\]
Therefore we can find $\ell_0>0$ large enough such that
\[
 \lambda_p(\mathcal L_{\ell_0}+f'(0))>f'(0)/2.
\]
Define
\[
\mathcal L^n_{\ell}\,[\phi](x):=d\int_{-\ell}^{\ell}J_n(x-y)\phi(y)dy-d\phi(x),
\]
\[
\widetilde {\mathcal L}^n_{\ell}\,[\phi](x):=d\sigma_n \int_{-\ell}^{\ell}\tilde J_n(x-y)\phi(y)dy-d\sigma_n\phi(x).
\]
Then it is readily checked that
\[
\lambda_p(\widetilde {\mathcal L}^n_{\ell_0}+f_n'(0))=\lambda_p({\mathcal L}^n_{\ell_0}+f'(0)+(d-1)(1-\sigma_n))\to \lambda_p(\mathcal L_{\ell_0}+f'(0))>f'(0)/2
\]
as $n\to\infty$. Therefore we can find $n_0>0$ large enough so that
\[
\lambda_p(\widetilde {\mathcal L}^n_{\ell_0}+f_n'(0))>f'(0)/4
\mbox{ for all } n\geq n_0.
\]
Using this fact we see from the proof of Theorem 1.3 in \cite{CDLL2019} that spreading happens to \eqref{main-n-tilde}
with $n\geq n_0$, provided that there exists $T\geq 0$ such that
\begin{equation}\label{2ell0}
h_n(T)-g_n(T)>2\ell_0.
\end{equation}

We show next that there exists $T>0$ large so that \eqref{2ell0} holds for all large $n$, say $n\geq n_1\geq n_0$.
Indeed, since spreading happens to \eqref{main-fb} by assumption, we can find $T>0$ large enough such that
$h(T)-g(T)>4\ell_0$.

Since $(d\sigma_n, J_n, f_n)\to (d, J, f)$ as $n\to\infty$ in the obvious sense, by the continuous dependence of the solution of \eqref{main-fb} on $(d, J, f)$ (which follows easily from the uniqueness of the solution), we see that, as $n\to\infty$, the solution $(u_n, g_n, h_n)$ of \eqref{main-n-tilde} converges to the solution $(u, g, h)$ of \eqref{main-fb} over any bounded time interval; in particular,
\[
g_n(t)\to g(t),\; \; h_n(t)\to h(t) \mbox{ uniformly for } t\in [0, T].
\]
It follows that
\[
h_n(T)-g_n(T)>\frac 12 [h(T)-g(T)]>2\ell_0 \mbox{ for all large } n.
\]
This proves \eqref{2ell0} and hence spreading happens to $(u_n, g_n, h_n)$ for every large $n$.
\end{proof}

Since each $\tilde J_n$ satisfies {\bf (J)} and {\bf (J1)}, we are in a position to apply part (i) of Theorem \ref{thm1} to \eqref{main-n-tilde} to conclude that
\begin{equation}\label{g-h-n}
\lim_{t\to\infty}\frac{h_n(t)}{t}=\lim_{t\to\infty}\frac{-g_n(t)}{t}=c_n
\end{equation}
for every large $n$, and $c_n>0$ is determined by the following two equations
\begin{equation}\label{semi-wave-n}
\begin{cases}
\displaystyle d \int_{-\infty}^0 J_n(x-y) \phi(y) dy - d \phi(x)+ c_n\phi'(x) + f(\phi(x)) =0, &  -\infty < x< 0,\\
\displaystyle \phi(-\infty) = \eta_n,\ \ \phi(0) =0,
\end{cases}
\end{equation}
and
\begin{equation}\label{fby-n}
c_n=\mu\int_{-\infty}^{0}\int_{0}^{+\infty}J_n(x-y)\phi(x)dydx,
\end{equation}
namely $c_n>0$ is the unique value such that \eqref{semi-wave-n} and \eqref{fby-n} have a solution $\phi=\phi_n\in C^1((-\infty, 0])$ which is strictly decreasing in $x$.
 This last fact is a consequence of Theorem \ref{thm2} applied to \eqref{main-n-tilde}, but with the corresponding equations of \eqref{semi-wave} and \eqref{fby-0}, which now involve $(d\sigma_n, \tilde J_n, f_n)$,
rewritten in terms of $(d, J_n, f)$, much as in the equivalent form \eqref{main-n} of \eqref{main-n-tilde}.

By Lemma \ref{lem4.1} and \eqref{g-h-n}, we have $c_n\leq c_{n+1}$ and
\[
\liminf_{t\to\infty}\frac{h(t)}{t}\geq c_n,\; \liminf_{t\to\infty}\frac{-g(t)}{t}\geq c_n \mbox{ for all large } n.
\]
Therefore part (ii) of Theorem \ref{thm1} is a consequence of the following lemma.

\begin{lemma}\label{lem4.3}
 $\lim_{n\to\infty} c_n=\infty$.
\end{lemma}
\begin{proof}
Arguing indirectly we assume that the conclusion of the lemma does not hold. Then the nondecreasing positive sequence $c_n$
must converge to some positive constant $c_\infty$ as $n\to\infty$. We show next that this leads to a contradiction.

We note that \eqref{semi-wave-n} and \eqref{fby-n} have a solution $\phi_n$ which is strictly decreasing. So we are in a position to argue as in the proof of Theorem \ref{thm-semiwave-existence} (with simple minor changes) to conclude that, either \eqref{tw-cauchy} has a nonincreasing solution $\phi(x)$ for $c=c_\infty$, or \eqref{semi-wave} and \eqref{fby-0} have a solution pair $(c,\phi)$ with $c=c_\infty$ and $\phi(x)$ nonincreasing in $x$. Since $J$ does not satisfy {\bf (J1)} and hence also does not satisfy {\bf (J2)}, we have a contradiction to either Proposition \ref{tw} or Theorem \ref{thm2}. This completes the proof.
\end{proof}

\section{Limiting profile as $\mu\rightarrow+\infty$}

In this section, we  discuss the convergence of the semi-wave pair $(c,\phi)$ when $\mu\rightarrow +\infty$ under the conditions {\bf (J)}, {\bf (J1)} and {\bf (f3)}. We will also show that \eqref{Cauchy} can be viewed as the limiting problem of \eqref{main-fb} as $\mu\to+\infty$. For this latter conclusion, we
only require {\bf (J)} and {\bf (f1)-(f2)}.

 It is obvious that the semi-wave pair $(c, \phi)$ depends on $\mu$. To stress this dependence, we will denote it by $(c_\mu, \phi_\mu)$ throughout this section.
 For each fixed $\mu>0$, we know that $\phi_\mu(x)$ is strictly decreasing in $(-\infty, 0]$, and hence there exists a unique $l_\mu>0$ such that
 $
 \phi_\mu(-l_\mu)=\frac 12.
 $
Define
\[
\hat \phi_\mu(x):=\phi_\mu(x-l_\mu) \mbox{ for } x\leq l_\mu, \mbox{ and so } \hat\phi_\mu(0)=\frac 12.
\]
Let us also define $\phi_\mu(x)=0$ for $x>0$.

 Firstly, we consider the case that condition {\bf (J2)} holds.

\begin{theorem}\label{thm5.1}
Suppose that {\bf (J)}, {\bf (J2)} and {\bf (f3)} are satisfied. Then, as $\mu\rightarrow +\infty$,
$$c_\mu\rightarrow c_*, \;\;   l_\mu\rightarrow +\infty, \; \phi_\mu(x)\to 0 \mbox{ and }\hat \phi_\mu(x)\rightarrow\phi_*(x) \mbox{ locally uniformly in $\mathbb{R}$, } $$
where $(c_*,\phi_*)$ is the minimal speed solution pair of \eqref{tw-cauchy} with $\phi_*(0)=1/2$.
\end{theorem}
\begin{proof}
Choose $u_0$ such that spreading happens to \eqref{main-fb}, and let $(u_\mu, g_\mu, h_\mu)$ denote the unique solution of \eqref{main-fb}. Let $u$ denote the unique solution of \eqref{Cauchy} with the same initial function $u_0$ (extended by 0 outside  $[-h_0, h_0]$). Then the comparison principle infers that
$u_\mu(t,x)\leq u(t,x)$ for $t>0,\; x\in [g (t), h(t)]$. From the proof of Lemma \ref{>c0} we see that $u_\mu(t,x)\geq \underline u(t-T, x)$ for $x\in [-\underline h(t-T), \underline h(t-T)]$ for all $t>T$. Therefore
\[\mbox{
$u(t,x)\geq \underline u(t-T, x)$ for $x\in [-\underline h(t-T), \underline h(t-T)]$ for all $t>T$.}
\]
From this and \eqref{c*} we easily deduce $c_\mu\leq c_*$ for all $\mu>0$. By the comparison principle, $h_\mu(t)$ is increasing in $\mu$, which implies that $c_\mu$ is nondecreasing in $\mu$. Therefore
\[
c_\infty:=\lim_{\mu\to+\infty} c_\mu \mbox{ exists, and } c_\infty\leq c_*.
\]

We are now ready to show that $\lim_{\mu\to+\infty} l_\mu=+\infty$. Indeed,
since
\begin{equation}\label{5.1}
0\leq \int_{-\infty}^0\int_0^{+\infty} J(x-y)\phi_\mu(x)dydx=\frac{c_\mu}{\mu}\leq \frac{c_*}\mu,
\end{equation}
and $\phi_\mu(x)$ is strictly decreasing in $x$, in the case that $J$ does not have compact support,
we must have $\lim_{\mu\to+\infty}\phi_\mu(x)\to 0$  locally uniformly in $(-\infty, 0]$, which immediately implies  $l_\mu\to+\infty$.
If $J$ has compact support, and $L:=\inf\{x>0: J(x)=0\}$, then \eqref{5.1} implies
\begin{equation}\label{5.2}
\lim_{\mu\to+\infty} \phi_\mu(x)=0 \mbox{ uniformly for } x\in [-L, 0].
\end{equation}
We show that in this case we also have $\lim_{\mu\to+\infty}\phi_\mu(x)\to 0$  locally uniformly in $(-\infty, 0]$. Indeed, from the equation \eqref{semi-wave} and the monotonicity of $c_\mu>0$, it is easily seen that $\phi_\mu'$ is  uniformly bounded for $\mu>1$ and $x\in (-\infty, 0]$. Therefore,
for any sequence $\mu_n\to+\infty$, $\phi_{\mu_n}$ has a subsequence, still denoted by itself for convenience of notation, such that $\phi_{\mu_n}$ converges to some $\phi_\infty$ locally uniformly in $(-\infty, 0]$. Moreover,  $\phi_\infty(x)$ is nonincreasing in $x$, is $C^1$ for $x\leq 0$, and satisfies
\begin{equation}\label{abc}
d \int_{-\infty}^0 J(x-y) \phi_\infty(y) dy- d\phi_\infty(x) + c_\infty \phi_\infty'(x) + f(\phi_\infty(x)) =0 \mbox{ for } x\leq 0.
\end{equation}
It suffices to show that $\phi_\infty\equiv 0$. From \eqref{5.2} we have $\phi_\infty(x)=0$ for $x\in [-L, 0]$. If $\phi_\infty\not\equiv 0$,
then by its monotonicity there exists $L_0\leq -L$ such that
\[
\phi_\infty(x)=0 \mbox{ in } [L_0, 0],\; \phi_\infty(x)>0 \mbox{ in } (-\infty, L_0).
\]
 It follows that $\phi_\infty'(L_0)=0$.
On the other hand, from \eqref{abc} and the definition of $L_0$ we also have
\[
c_\infty \phi_\infty'(L_0)=-d\int_{-\infty}^0J(L_0-y)\phi_\infty(y)dy<0.
\]
This contradiction shows that we must have $\phi_\infty\equiv 0$ and hence we always have $\lim\limits_{\mu\to+\infty}l_\mu=+\infty$.

We may now use the equation \eqref{semi-wave} and the monotonicity of $c_\mu>0$ to see that $\hat\phi_\mu'$ is  uniformly bounded for $\mu>1$ and $x\in (-\infty, l_\mu]$. Repeating the argument of the last paragraph we can conclude that, for any sequence $\mu_n\to+\infty$, $\hat\phi_{\mu_n}$ has a subsequence, still denoted by itself, such that $\hat\phi_{\mu_n}$ converges to some $\hat\phi_\infty$ locally uniformly in $\mathbb R$, and  $\hat\phi_\infty(x)$ is nonincreasing in $x$, $\hat \phi_\infty(0)=\frac 12$, and
\begin{equation}\label{abc1}
d \int_{\mathbb R} J(x-y) \hat\phi_\infty(y) dy- d\hat\phi_\infty(x) + c_\infty \hat\phi_\infty'(x) + f(\hat\phi_\infty(x)) =0 \mbox{ for } x\in\mathbb R.
\end{equation}
Obviously, $\hat\phi_\infty\not\equiv 1/2$, for otherwise $\hat\phi_\infty$ does not satisfy \eqref{abc1}. Since $0$ and $1$ are the only nonnegative zeros of $f$ under {\bf (f3)}, we necessarily  have $\hat\phi_\infty(-\infty)=1$ and $\hat\phi_\infty(+\infty)=0$. The strong maximum principle then infers $\hat\phi_\infty>0$ in $\mathbb R$ (for example, we may apply Lemma \ref{lm-strongMP} with $x<0$ replaced by $x<l$ for any $l>0$).
 Therefore, $(c_\infty,\hat\phi_\infty)$ is a solution of \eqref{tw-cauchy}. Since $c_*$ is the minimal speed of \eqref{tw-cauchy} and $c_\infty\leq c_*$, we  necessarily have $c_*=c_\infty$ and $\phi_*=\hat\phi_\infty$. The uniqueness of $\phi_*$ implies that $\hat\phi_{\mu}$ converges to $\phi_*$ locally uniformly in $\mathbb R$ as $\mu\to +\infty$.
\end{proof}

Next, let us consider the case {\bf (J1)} holds but {\bf (J2)} does not.
\begin{theorem}\label{thm5.2}
Suppose that {\bf (J)}, {\bf (J1)} and {\bf (f3)} are satisfied but {\bf (J2)} does not hold. Then $c_\mu\rightarrow +\infty$ as $\mu\rightarrow +\infty$.

\end{theorem}
\begin{proof}
As before, the comparison principle implies that $c_\mu$  increases in $\mu$, and so we can define $c_\infty:=\lim_{\mu\to +\infty}c_\mu\in (0, +\infty]$.
If  $c_\infty<+\infty$, then we can repeat the argument in the proof of Theorem \ref{thm5.1} to conclude that \eqref{abc1} has a solution pair $(c_\infty, \hat\phi_\infty)$ with
$\hat\phi_\infty>0$, $\hat\phi_\infty(-\infty)=1$, $\hat\phi_\infty(+\infty)=0$ and $\hat\phi_\infty(0)=1/2$.
Thus $(c_\infty,\hat\phi_\infty)$ is a solution of \eqref{tw-cauchy} with finite speed, which is a contradiction with the assumption that  {\bf (J2)} is not satisfied. Therefore we necessarily have $c_\infty=+\infty$.
\end{proof}

Finally, let us discuss the relationship between the solution of \eqref{main-fb} and that of \eqref{Cauchy}.
\begin{theorem}\label{thm5.3}
Suppose that {\bf (J)} and {\bf (f1)-(f2)} are satisfied. Let $(u_\mu, g_\mu, h_\mu)$ be the solution of \eqref{main-fb} for $\mu>0$ with initial datum $u_0$, and $u_*$ be the solution of \eqref{Cauchy} with the same initial datum {\rm (}extended by zero outside $[-h_0, h_0]${\rm)}. Then $-g_\mu, h_\mu\to +\infty$ locally uniformly in $(0, +\infty)$ and $u_\mu\rightarrow u_*$ locally uniformly in $\mathbb{R}^+\times\mathbb{R}$ as $\mu\rightarrow +\infty$.
\end{theorem}
\begin{proof}
By the comparison principle we know that $-g_\mu, h_\mu$ and $u_\mu$ are increasing in $\mu$. Therefore for each $t>0$,
\[
g_\infty(t):=\lim_{\mu\to+\infty} g_\mu(t)\in [-\infty, -h_0),\; \; h_\infty(t):=\lim_{\mu\to+\infty} h_\mu(t)\in (h_0, +\infty],
\]
\[ \mbox{ and }\;\;\;  u_\infty(t,x):=\lim_{\mu\to+\infty} u_\mu(t,x),\;\; g_\infty(t)<x<h_\infty(t),\; t>0,
\]
are well-defined. Moreover,
\[
0<u_\infty(t,x)\leq M_0^*:=\max\{\|u_0\|_\infty, K_0\} \mbox{ for } t>0,\; g_\infty(t)<x<h_\infty(t),
\]
where $K_0$ appears in {\bf (f2)}. Since $J(0)>0$ there exists $\delta>0$ small so that
\[
\sigma_0:=\min_{|x|\leq 2\delta } J(x)>0.
\]

If there exists $t_0>0$ such that $h_\infty(t_0)<+\infty$, then clearly $h_\mu(t)\leq h_\infty(t)\leq h_\infty (t_0)<+\infty$ for $t\in (0, t_0]$, $\mu>0$, and
\[
\begin{aligned}
h_\mu'(t)&\geq \mu\int_{h_{\mu}(t)-\delta}^{h_{\mu}(t)}\int_{h_{\mu}(t)}^{h_\mu(t)+\delta} J(x-y)u_{\mu}(t,x)dydx\\
&\geq \mu \sigma_0\delta\int_{h_{\mu}(t)-\delta}^{h_{\mu}(t)} u_{\mu}(t,x)dx \mbox{ for }  \mu>0,\; t>0.
\end{aligned}
\]
Denote
\[
m(t,\mu):=\int_{h_{\mu}(t)-\delta}^{h_{\mu}(t)} u_{\mu}(t,x)dx.
\]
Then clearly $m\in L^\infty$ and by the dominated convergence theorem we have
\[
\lim_{\mu\to+\infty} m(t,\mu)=m_\infty(t):=\int_{h_{\infty}(t)-\delta}^{h_{\infty}(t)} u_{\infty}(t,x)dx>0 \mbox{ for } t\in (0, t_0].
\]
The dominated convergence theorem then yields
\[
[h_\mu(t_0)-h_0]/\mu\geq \sigma_0\delta \int_0^{t_0} m(t,\mu)dt\to \sigma_0\delta\int_0^{t_0}m_\infty(t)dt>0 \mbox{ as } \mu\to+\infty.
\]
It follows that
$
\lim_{\mu\to+\infty}h_\mu(t_0)= +\infty$, a contradiction to the assumption $h_\infty(t_0)<+\infty$. Therefore $h_\infty(t)=+\infty$ for every $t>0$, which, together with the monotonicity of $h_\mu(t)$ in $t$,  implies
\[
\mbox{ $\lim_{\mu\to+\infty}h_\mu(t)= +\infty$ locally uniformly in $(0, \infty)$.}
\]
We can similarly show that
\[
\mbox{ $\lim_{\mu\to+\infty}g_\mu(t)= -\infty$ locally uniformly in $(0, \infty)$.}
\]

We now examine $u_\infty$. 
From \eqref{main-fb} we obtain
\begin{equation}
\label{u_mu}
u_\mu(t,x)=u_\mu(t_0,x)+\int_{t_0}^t\Big[d\int_{g_\mu(\tau)}^{h_\mu(\tau)}J(x-y)u_\mu(\tau, y)dy
-du_\mu(\tau, x)+f(u_\mu(\tau,x))\Big]d\tau,
\end{equation}
where $0<t_0<t$ and $g_\mu(t_0)<x<h_\mu(t_0)$.

Letting $\mu\to\infty$ in \eqref{u_mu}, and using $u_\mu\to u_\infty$, $g_\mu\to-\infty$, $h_\mu\to +\infty$,  the uniform boundedness of $u_\mu$, and $J\in L^1(\mathbb R)$, we deduce
\begin{equation}\label{u_infty}
u_\infty(t,x)=u_\infty(t_0,x)+\int_{t_0}^t\Big[d\int_{-\infty}^\infty J(x-y)u_\infty(\tau, y)dy-du_\infty (\tau,x)+f(u_\infty(\tau, x))\Big]d\tau.
\end{equation}

 We next show that $\lim_{t\to 0} u_\infty(t,x)=u_0(x)$ uniformly for $x\in\mathbb R$. Note that, by the comparison principle we have
 \[
 u_{\mu_0}(t,x)\leq u_\mu(t,x)\leq u_*(t,x) \mbox{ for } t>0,\; x\in [-h_0, h_0],\; \mu>\mu_0.
 \]
 Letting $\mu\to+\infty$ we immediately obtain
 \[
 u_{\mu_0}(t,x)\leq u_\infty(t,x)\leq u_*(t,x) \mbox{ for } t>0,\; x\in [-h_0, h_0].
 \]
 Letting $t\to 0$, we obtain from the above inequalities that $u_\infty(t,x)\to u_0(x)$ uniformly for $x\in [-h_0, h_0]$.

 For $x\in\mathbb R\setminus [-h_0, h_0]$, we may use the inequalities
 \[
0\leq  u_\mu(t,x)\leq u_*(t,x) \mbox{ for } t>0,\; x\in [g_\mu(t), h_{\mu}(t)]
 \]
 to obtain
 \[
 0\leq  u_\infty(t,x)\leq u_*(t,x) \mbox{ for } t>0,\; x\in\mathbb R.
 \]
 Then letting $t\to 0$ we deduce
 \[
 \lim_{t\to\infty} u_\infty(t,x)=0=u_0(x) \mbox{ uniformly for } x\in\mathbb R\setminus [-h_0,h_0].
 \]
 Therefore we may let $t_0\to 0$ in \eqref{u_infty} to obtain
 \[
 u_\infty(t,x)=u_0(x)+\int_{0}^t\Big[d\int_{-\infty}^\infty J(x-y)u_\infty(\tau, y)dy-du_\infty (\tau,x)+f(u_\infty(\tau, x))\Big]d\tau,
 \]
 for $t>0$, $x\in\R$.
 This clearly implies that
 $u_\infty$ satisfies  \eqref{Cauchy}, and hence $u_\infty\equiv u_*$ by uniqueness. 
 
 We next prove
 that $u_\mu\to u_\infty=u_*$
locally uniformly in $\mathbb R^+\times\mathbb R$, namely,
given any finite interval $I_0=[\sigma_0, M_0]\subset (0,+\infty)$,
\[
\lim_{\mu\to\infty} W_\mu(t,x)=0 \mbox{ uniformly for } (t,x)\in I_0\times I_0,
\]
where $W_\mu(t,x):=u_\mu(t,x)-u_\infty(t,x)$.

For any given $\epsilon>0$, due to $u_{\mu_0}\leq u_\mu\leq u_\infty$ for $\mu\geq \mu_0:=1$, and $u_{\mu_0}, u_\infty \to u_0$ as $t\to 0$ uniformly in $x\in\R$, we can find $t_0\in (0,\sigma_0)$ such that
\[
|W_\mu(t_0,x)|=|u_\mu(t_0,x)-u_\infty(t_0,x)|<\epsilon \mbox{ for } x\in\R,\; \mu\geq \mu_0.
\]
We can then find $M>M_0$ and $\mu_1>\mu_0$ such that, for $(t,x)\in [t_0, M_0]\times I_0$ and $\mu\geq\mu_1$, 
\[
[g_\mu(t), h_\mu(t)]\supset [-M, M],\; E_\mu(t,x):=d\int_{\R\setminus [g_\mu(t), h_\mu(t)]}J(x-y)W_\mu(t,y)dy\in (-\epsilon, \epsilon).
\]
From the equations for $u_\mu$ and $u_\infty$ we obtain
\[
\partial_t W_\mu(t,x)=d\int_{\R}J(x-y)W_\mu(t, y)dy+\beta_\mu(t,x)W_\mu(t, x)-E_\mu(t,x),
\]
where $\beta_\mu(t,x):=-d$ if $u_\mu(t,x)-u_\infty(t,x)=0$, and otherwise
\[
\beta_\mu(t,x):=\frac{f(u_\mu(t,x))-f(u_\infty(t,x))}{u_\mu(t,x)-u_\infty(t,x)}-d,
\]
and so $|\beta_\mu(t,x)|\leq \beta^*:=d+K(\|u_\infty\|_\infty])$.

Therefore, for $(t,x)\in [t_0, M_0]\times I_0$, we have
\begin{align*}
|W_\mu(t,x)|=&\,e^{\int_{t_0}^t\beta(s,x)ds}\left|W_\mu(t_0,x)+\int_{t_0}^te^{-\int_{t_0}^\tau \beta(s,x)ds}\Big[d\int_{\R}J(x-y)W_\mu(\tau,y)dy-E_\mu(\tau,x)\Big]d\tau\right|\\
\leq& \, e^{M_0\beta^*}\left[\epsilon+e^{M_0\beta^*}\int_{t_0}^t\Big[d\int_{\R} J(x-y)|W_\mu(\tau,y)|dy+\epsilon\Big]d\tau\right].
\end{align*}
Since $W_\mu(\tau,y)\to 0$ as $\mu\to\infty$, and $|W_\mu| $ is uniformly bounded, and $J\in L^1(\R)$, we have
\[
\lim_{\mu\to\infty} \int_{t_0}^{M_0}\int_{\R} J(x-y)|W_\mu(\tau,y)|dyd\tau=0.
\]
It follows that
\[
\limsup_{\mu\to\infty} |W_\mu(t,x)|\leq e^{M_0\beta^*}\left[\epsilon+M_0e^{M_0\beta^*}\epsilon\right]
\mbox{ uniformly for } (t,x)\in [t_0, M_0]\times I_0\supset I_0\times I_0.
\]
Letting $\epsilon\to 0$ we deduce
\[
\lim_{\mu\to\infty} |W_\mu(t,x)|=0 
\mbox{ uniformly for } (t,x)\in  I_0\times I_0,
\]
  as desired.
\end{proof}


\section*{Acknowledgments}
\noindent

Y. Du was partially supported by the Australian Research Council. F. Li was partially supported by NSF  of China (11431005, 11971498). Y. Du would like to thank Dr Wenjie Ni for suggesting the current proof of Lemma 2.2.

\end{document}